\documentclass[11pt]{elsarticle}

\usepackage{amssymb}
\usepackage{xcolor}
\usepackage{mathrsfs}
\usepackage{amscd}
\usepackage{appendix}
\usepackage{geometry}
\usepackage{setspace}
\usepackage{amsfonts}
\usepackage{amsmath}
\usepackage{lineno,hyperref}
\usepackage{tikz}
\usetikzlibrary{matrix,arrows}

\usepackage{amssymb,amscd,amsthm,verbatim}
\usepackage{bbm}

\date{\today}

\newtheorem{theorem}{Theorem}[section]
\newtheorem{lemma}[theorem]{Lemma}
\newtheorem{prop}[theorem]{Proposition}
\newtheorem{coro}[theorem]{Corollary}

\theoremstyle{definition}

\newtheorem*{definition 1}{Definition 1}
\newtheorem*{definition 2}{Definition 2}
\newtheorem*{definition 3}{Definition 3}
\newtheorem*{definition 4}{Definition 4}
\newtheorem*{definition 5}{Definition 5}

\newtheorem{remark}[theorem]{Remark}

\theoremstyle{plain}

\allowdisplaybreaks \numberwithin{equation}{section}

\journal{arXiv}

\begin{document}
\title{The spectrum of the multi-frequency quasi-periodic CMV matrices contains intervals}

\author[]{Bei Zhang}
\ead{beizhangmath@aliyun.com}

\author[]{Daxiong Piao\corref{cor1}}
\ead{dxpiao@ouc.edu.cn}
\address{School of Mathematical Sciences,  Ocean University of China, Qingdao 266100, P.R.China}

\cortext[cor1]{Corresponding author}

\begin{abstract}
We investigate the spectral structure of multi-frequency quasi-periodic CMV matrices with Verblunsky coefficients defined by shifts on the $d$-dimensional
torus. Under the positive Lyapunov exponent regime and standard Diophantine frequency conditions, we establish that the spectrum of these operators contains intervals on the unit circle.

\end{abstract}

\maketitle


\section{Introduction}
Recently, one-frequency CMV matrices with distinct Verblunsky coefficients have attracted significant attention and substantial progress has been achieved in multiple areas, including the positivity of Lyapunov exponents \cite{Damanik-Kruger-2009,Zhang-Nonlinearity}, localization theory \cite{LPG24-arXiv,LPG23-JFA,WD19-JFA,Zhu-JAT}, and connections to quantum walks \cite{Cedzich-2021,WD19-JFA}. As the mathematical theory of CMV matrices continues to develop, it becomes natural to extend these investigations to multi-frequency settings. For small quasi-periodic Verblunsky coefficients with $d$-dimensional frequency, Li, Damanik, and Zhou \cite{Li-Damanik-Zhou} established the pure absolute continuity of the spectral measure on the essential spectrum for the corresponding CMV matrices. Zhang and Piao \cite{Zhang-Piao} formulated the finite and full scale localization for the multi-frequency quasi-periodic CMV matrices, which extended the existing results for Schr\"{o}dinger operators in \cite{GSV16-arXiv,GSV19-Inventiones}.


Despite these advances, the spectral structure--particularly the existence of interval spectra--remains poorly understood for multi--frequency quasi-periodic CMV matrices, owing to fundamental differences in operator dynamics compared to Schr\"{o}dinger-type operators. An important result, going back to Geronimo, Johnson and Simon \cite{GJ-1998,Simon-book2}, is that the finite band spectrum for the orthogonal polynomials on the unit circle (OPUC for short). To our best knowledge, there is no such result for the multi-frequency quasi-periodic (MF-QP for short) CMV matrices. Hence, the main object under consideration is the interval spectrum problem in this paper. For this purpose, our research hinges on several basic tools obtained in \cite{Zhang-Piao}.

Our main result is as follows:
\begin{theorem}\label{main-theorem}
In the positive Lyapunov exponent regime, analytic MF-QP CMV matrices with Verblunsky coefficients generated by a shift on the $d$-dimensional torus have spectra containing intervals on the unit circle, provided the frequencies satisfy the standard Diophantine condition.
\end{theorem}


\begin{remark} Let's make some comments on our results.

\item(a) Simon's inverse spectral theorem \cite[Theorem 11.4.7]{Simon-book2} establishes that for any finite union of intervals on the unit circle $\partial \mathbb{D}$, there exist quasi-periodic Verblunsky coefficients $\{\alpha_{j}\}_{j=0}^{\infty}$  such that the essential support of the spectral measure of the half-line CMV matrix coincides with these intervals. In contrast, our result demonstrates that for extended CMV matrices (i.e., doubly infinite matrices), the spectrum itself contains intervals under the MF-QP setting. This extends Simon's conclusion to the extended operator framework while focusing on the inverse spectral problem.

\item(b) For discrete Schr\"{o}dinger operators, analogous conjectures and results have been explored. Chulaevsky and Sina\v{\i} conjectured in \cite{CS89-CMP} that the spectrum of the two-frequency Schr\"{o}dinger operator generically forms an interval for large smooth potentials. Goldstein, Schlag, and Voda  \cite{GSV19-Inventiones} later proved this conjecture, showing that the spectrum of MF-QP  Schr\"{o}dinger operators with large coupling is indeed a single interval. Kr\"{u}ger \cite{Kruger12-JFA} extended this to skew-shift  Schr\"{o}dinger operators, while in \cite{Kruger13-IMRN}, he further studied OPUC with skew-shift Verblunsky coefficients. However, these results rely crucially on the real analyticity of potentials, a property not shared by CMV matrices due to the complex analytic nature of Verblunsky coefficients.

\item(c) Our work provides a CMV counterpart to the spectral interval results of \cite{GSV19-Inventiones} , but with fundamental differences arising from the operator structure:
\begin{enumerate}
  \item Analyticity Framework: The complex analytic nature of Verblunsky coefficients in CMV matrices (vs. real analytic Schr\"{o}dinger potentials) imposes phase-space constraints. This allows fragmented spectral configurations prohibited in Schr\"{o}dinger systems.

  \item Operator Class Disparity: As unitary operators on the unit circle (contrasting self-adjoint Schr\"{o}dinger operators on the real line), CMV matrices permit spectral gap formation under quasi-periodic perturbations, a phenomenon suppressed in connected spectra of self-adjoint frameworks.

  \item Spectral Interval Phenomenology: While both systems exhibit interval spectra in positive Lyapunov exponent regimes, the CMV spectrum demonstrates essential multi-interval characteristics. This divergence originates from the analytic continuation barriers inherent in complex coefficient systems.

\end{enumerate}
\end{remark}

The present paper is organized as follows. After recalling some preliminaries of the OPUC theory in Section 2, we introduce some basic tools in Section 3 which are helpful to the following research. In Section 4, we present four inductive conditions to investigate the structure of the spectrum for the MF-QP CMV matrices. We prove the main theorem in Section 5. For easier reading, we introduce some useful notations and lemmas in the Appendix.

\section{Preliminaries}
We begin by recalling key concepts of OPUC. Let $\mathbb{D}=\{z:|z|<1\}$ be the open unit disk in $\mathbb{C}$ and $\mu$ be a nontrivial probability measure on $\partial \mathbb{D}=\{z:|z|=1\}$.  Therefore, the functions $1, z, z^{2}, \ldots$ are linearly independent in the Hilbert space $L^{2}(\partial \mathbb{D}, d\mu)$. Let $\Phi_{n}(z)$ be the monic orthogonal polynomials; that is,
$\Phi_{n}(z)=P_{n}[z^{n}]$,  $P_{n}\equiv$ projection onto $\{1, z, z^{2}, \ldots, z^{n-1}\}^{\bot}$ in $L^{2}(\partial \mathbb{D}, d\mu)$. Then the orthonormal polynomials are
$$\varphi_{n}(z)=\frac{\Phi_{n}(z)}{\|\Phi_{n}(z)\|_{\mu}},$$
where $\|\cdot\|_{\mu}$ denotes the norm of $L^{2}(\partial \mathbb{D}, d\mu)$. For comprehensive foundations, we refer to Simon's treatise \cite{Simon-book}

Then the Verblunsky coefficients $\{\alpha_{n}\}_{n=0}^{\infty}$ with $\alpha_{n} \in \mathbb{D}$ for all $n\geq 0$, are uniquely determined by the Szeg\H{o} recursion:

\begin{equation}\label{Szego-1}
\Phi_{n+1}(z)=z\Phi_{n}(z)-\overline{\alpha}_{n}\Phi_{n}^{*}(z),
\end{equation}
where the Szeg\H{o} dual  $\Phi_{n}^{*}(z)$ is defined by $z^{n}\overline{\Phi_{n}(1/\overline{z})}.$

Orthonormalizing $1, z, z^{-1}, z^{2}, z^{-2},\cdots$, we can get a CMV basis $\{\chi_{j}\}_{j=0}^{\infty}$ of $L^{2}(\partial \mathbb{D},d\mu)$, and the matrix representation of multiplication by $z$ relative to the CMV basis gives rise to the CMV matrix $\mathcal{C}$, where
$$\mathcal{C}_{ij}=\langle \chi_{i},z\chi_{j} \rangle.$$
Then $\mathcal{C}$ has the form
\begin{equation*}
\mathcal{C}=
\left(
\begin{array}{ccccccc}
\overline{\alpha}_{0}&\overline{\alpha}_{1}\rho_{0}&\rho_{1}\rho_{0}&  &  & & \\
\rho_{0}&-\overline{\alpha}_{1}\alpha_{0}&-\rho_{1}\alpha_{0}& & & &\\
& \overline{\alpha}_{2}\rho_{1} &-\overline{\alpha}_{2}\alpha_{1}& \overline{\alpha}_{3}\rho_{2}&\rho_{3}\rho_{2}& & \\
& \rho_{2}\rho_{1}& -\rho_{2}\alpha_{1}&-\overline{\alpha}_{3}\alpha_{2}&-\rho_{3}\alpha_{2}& & \\
& & &\overline{\alpha}_{4}\rho_{3}&-\overline{\alpha}_{4}\alpha_{3}&\overline{\alpha}_{5}\rho_{4}& \\
& & &\rho_{4}\rho_{3}&-\rho_{4}\alpha_{3}&-\overline{\alpha}_{5}\alpha_{4}& \\
& & & & \ddots & \ddots & \ddots
\end{array}
\right),
\end{equation*}
where $\rho_{j}=\sqrt{1-|\alpha_{j}|^{2}}$, $j\geq 0$. Similarly, we can get the extended CMV matrix
\begin{equation*}
\mathcal{E}=
\left(
\begin{array}{cccccccc}
\ddots&\ddots&\ddots& & & & & \\
& -\overline{\alpha}_{0}\alpha_{-1}&\overline{\alpha}_{1}\rho_{0}&\rho_{1}\rho_{0}&  &  & & \\
& -\rho_{0}\alpha_{-1}&-\overline{\alpha}_{1}\alpha_{0}&-\rho_{1}\alpha_{0}& & & &\\
& & \overline{\alpha}_{2}\rho_{1} &-\overline{\alpha}_{2}\alpha_{1}& \overline{\alpha}_{3}\rho_{2}&\rho_{3}\rho_{2}& & \\
& & \rho_{2}\rho_{1}& -\rho_{2}\alpha_{1}&-\overline{\alpha}_{3}\alpha_{2}&-\rho_{3}\alpha_{2}& & \\
& & & &\overline{\alpha}_{4}\rho_{3}&-\overline{\alpha}_{4}\alpha_{3}&\overline{\alpha}_{5}\rho_{4}& \\
& & & &\rho_{4}\rho_{3}&-\rho_{4}\alpha_{3}&-\overline{\alpha}_{5}\alpha_{4}& \\
& & & & & \ddots & \ddots & \ddots
\end{array}
\right),
\end{equation*}
which is a special five-diagonal doubly-infinite matrix in the standard
basis of $L^{2}(\partial \mathbb{D}, d\mu)$ according to \cite[Subsection 4.5]{Simon-book} and \cite[Subsection 10.5]{Simon-book2}.

In this paper, we consider a sequence of Verblunsky coefficients generated by an analytic function $\alpha(\cdot):\mathbb{T}^{d}\rightarrow \mathbb{D}$, $\alpha_{n}(x)=\alpha(T^{n}x)=\alpha(x+n\omega)$, where $T$ is an invertible map of the form $Tx=x+\omega$, $\mathbb{T}:=\mathbb{R}/\mathbb{Z}$, $x, \omega\in \mathbb{T}^{d}$ are  phase and frequency, respectively. Moreover, the frequency vector $\omega\in \mathbb{T}^{d}$ satisfies the standard Diophantine condition
\begin{equation}\label{standard-DC}
\|k\cdot \omega\|\geq \frac{p}{|k|^{q}}
\end{equation}
for all nonzero $k\in \mathbb{Z}^{d}$, where $p>0$, $q>d$ are some constants, $\|\cdot\|$ denotes the distance to the nearest integer and $|\cdot|$ stands for the sup-norm on $\mathbb{Z}^{d}$ ($|k|=|k_{1}|+|k_{2}|+\cdots+|k_{d}|$, $k_{i}$ is the $i$-th element of the vector $k$). For the sake of convenience, we denote $\mathbb{T}^{d}(p,q)\subset \mathbb{T}^{d}$ be the set of $\omega$ satisfying \eqref{standard-DC}.

Assume that the sampling function $\alpha(x)$ satisfies that
\begin{equation}\label{sampling-function}
\int_{\mathbb{T}^{d}}\log(1-|\alpha(x)|)dx>-\infty
\end{equation}
and it can extend complex analytically to
\begin{equation*}
\mathbb{T}^{d}_{h}:=\{x+iy:x\in \mathbb{T}^{d}, y\in \mathbb{R}^{d}, |y|<h\}
\end{equation*}
with some $h>0$.

Then the Szeg\H{o} recursion is equivalent to
\begin{equation}\label{Szego-2}
\rho_{n}(x)\varphi_{n+1}(z)=z\varphi_{n}(z)-\overline{\alpha}_{n}(x)\varphi_{n}^{*}(z),
\end{equation}
where $\rho_{n}(x)=\rho(x+n\omega)$ and $\rho(x)=(1-|\alpha(x)|^{2})^{1/2}$. Applying Szeg\H{o} dual to both sides of equation (\ref{Szego-2}), one can obtain that
\begin{equation}\label{Szego-3}
\rho_{n}(x)\varphi_{n+1}^{*}(z)=\varphi_{n}^{*}(z)-\alpha_{n}(x)z\varphi_{n}(z).
\end{equation}

Then equations (\ref{Szego-2}) and (\ref{Szego-3}) can be written as
\begin{equation*}
\left(
\begin{array}{c}
\varphi_{n+1}\\
\varphi_{n+1}^{*}
\end{array}
\right)
=
S(\omega,z;x+n\omega)
\left(
\begin{array}{c}
\varphi_{n}\\
\varphi_{n}^{*}
\end{array}
\right),
\end{equation*}
where
\begin{equation*}
S(\omega,z;x)=\frac{1}{\rho(x)}
\left(
\begin{array}{cc}
z&-\overline{\alpha}(x)\\
-\alpha(x)z&1
\end{array}
\right).
\end{equation*}
Since $\det S(\omega,z;x)=z$, one always prefers to study the following determinant $1$ matrix:
\begin{equation*}
M(\omega,z;x)=\frac{1}{\rho(x)}
\left(
\begin{array}{cc}
\sqrt{z}&-\frac{\overline{\alpha}(x)}{\sqrt{z}}\\
-\alpha(x)\sqrt{z}&\frac{1}{\sqrt{z}}
\end{array}
\right)\in \mathbb{S}\mathbb{U}(1,1),
\end{equation*}
which is called the Szeg\H{o} cocycle map. Then the monodromy matrix (or $n$-step transfer matrix) is defined by
\begin{equation}\label{n-step}
M_{n}(\omega,z;x)=\prod_{j=n-1}^{0}M(\omega,z;x+j\omega)
\end{equation}
According to (\ref{n-step}), it is obvious that
\begin{equation*}
M_{n_{1}+n_{2}}(\omega,z;x)=M_{n_{2}}(\omega,z;x+n_{1}\omega)M_{n_{1}}(\omega,z;x)
\end{equation*}
and
\begin{equation}\label{log-subadditive}
\log\|M_{n_{1}+n_{2}}(\omega,z;x)\|\leq \log\|M_{n_{1}}(\omega,z;x)\|+\log\|M_{n_{2}}(\omega,z;x+n_{1}\omega)\|.
\end{equation}
Let
\begin{equation}\label{un}
u_{n}(\omega,z;x):=\frac{1}{n}\log\|M_{n}(\omega,z;x)\|
\end{equation}
and
\begin{equation}\label{un}
L_{n}(\omega,z):=\int_{\mathbb{T}^{d}}u_{n}(\omega,z;x)dx.
\end{equation}

Integrating the inequality (\ref{log-subadditive}) with respect to $x$ over $\mathbb{T}^{d}$, we have that
\begin{equation*}
L_{n_{1}+n_{2}}(\omega,z)\leq \frac{n_{1}}{n_{1}+n_{2}}L_{n_{1}}(\omega,z)+\frac{n_{2}}{n_{1}+n_{2}}L_{n_{2}}(\omega,z).
\end{equation*}
This implies that
\begin{equation*}
L_{n}(\omega,z)\leq L_{m}(\omega,z) \quad \mathrm{if} \quad m<n,\; m|n
\end{equation*}
and
\begin{equation*}
L_{n}(\omega,z)\leq L_{m}(\omega,z)+C\frac{m}{n} \quad \mathrm{if} \quad m<n.
\end{equation*}

For any irrational $\omega\in \mathbb{T}^{d}$, the transformation $x\rightarrow x+\omega$ is ergodic. Notice that inequality (\ref{log-subadditive}) implies that $\log \|M_{n}(\omega,z;x)\|$ is subadditive. Thus, according to Kingman's subadditive ergodic theorem, the limit
\begin{equation}\label{Lyapunov-exponent}
L(\omega,z)=\lim_{n\rightarrow\infty}L_{n}(\omega,z)
\end{equation}
exists. This is called the Lyapunov exponent. Throughout this paper, we let $\gamma$ be the lower bound of the Lyapunov exponent. On the other hand, Furstenberg-Kesten theorem indicates that the limit also exists for $ a.e.\;x$:
\begin{equation}\label{Lyapunov-exponent-1}
\lim_{n\rightarrow\infty}u_{n}(\omega,z;x)=\lim_{n\rightarrow\infty}L_{n}(\omega,z)=L(\omega,z).
\end{equation}
Note that
\begin{equation}\label{log-Mn}
0\leq\log \|M_{n}(\omega,z;x)\|\leq C(\alpha_{i},z)n.
\end{equation}
Thus, we have that
\begin{equation}\label{Ln-bound}
0\leq L_{n}(\omega,z)\leq C(\alpha_{i},z)
\end{equation}
for $i\in \mathbb{Z}$.

Define the unitary matrices
\begin{equation*}
\Theta_{n}=
\left(
\begin{array}{cc}
\overline{\alpha}_{n}&\rho_{n}\\
\rho_{n}&-\alpha_{n}
\end{array}
\right).
\end{equation*}
Then one can factorize the matrix $\mathcal{C}$ as
\begin{equation*}
\mathcal{C}=\mathcal{L}_{+}\mathcal{M}_{+},
\end{equation*}
where
\begin{equation*}
\mathcal{L}_{+}=\left(
\begin{matrix}
\Theta_0 &~ & ~\\
~& \Theta_2 & ~\\
~ & ~& \ddots
\end{matrix}
\right),\quad
\mathcal{M}_{+}=\left(
\begin{matrix}
\mathbf{1} &~ & ~\\
~& \Theta_1 & ~\\
~ & ~& \ddots
\end{matrix}
\right).
\end{equation*}

Similarly, the extended CMV matrix can be written as
\begin{equation*}
\mathcal{E}=\mathcal{L}\mathcal{M},
\end{equation*}
with $\mathcal{L}$ and $\mathcal{M}$ defined by
\begin{equation*}
\mathcal{L}=\bigoplus_{j\in \mathbb{Z}}\Theta_{2j} \quad {\rm and} \quad \mathcal{M}=\bigoplus_{j\in \mathbb{Z}}\Theta_{2j+1}.
\end{equation*}

We let $\mathcal{E}_{[a,b]}$ denote the restriction of an extended CMV matrix to a finite interval $[a,b]\subset\mathbb{Z}$ defined by
\begin{equation*}
\mathcal{E}_{[a,b]}=P_{[a,b]}\mathcal{E}(P_{[a,b]})^{*},
\end{equation*}
where $P_{[a,b]}$ is the projection $\ell^{2}(\mathbb{Z})\rightarrow \ell^{2}([a,b])$. $\mathcal{L}_{[a,b]}$ and $\mathcal{M}_{[a,b]}$ are defined similarly.

However, the matrix $\mathcal{E}_{[a,b]}$ will no longer be unitary due to the fact that $|\alpha_{a-1}|<1$ and $|\alpha_{b}|<1$. To solve this issue, we need to modify the boundary conditions. With $\beta, \eta \in \partial \mathbb{D}$, we define the sequence of Verblunsky coefficients
\begin{equation*}
\tilde{\alpha}_{n}=
\begin{cases}
\beta, \quad & n =a-1;\\
\eta,& n =b;\\
\alpha_{n},& n \notin  \{a-1,b\}.
\end{cases}
\end{equation*}
Denote the extended CMV matrix with Verblunsky coefficients $\tilde{\alpha}_{n}$ by $\tilde{\mathcal{E}}$. Define
\begin{equation*}
\mathcal{E}_{[a,b]}^{\beta,\eta}=P_{[a,b]}\tilde{\mathcal{E}}(P_{[a,b]})^{*}.
\end{equation*}
$\mathcal{L}_{[a,b]}^{\beta,\eta}$ and $\mathcal{M}_{[a,b]}^{\beta,\eta}$ are defined correspondingly. Then $\mathcal{E}_{[a,b]}^{\beta,\eta}$, $\mathcal{L}_{[a,b]}^{\beta,\eta}$ and $\mathcal{M}_{[a,b]}^{\beta,\eta}$ are all unitary.

For $z\in \mathbb{C}$, $\beta, \eta\in\partial \mathbb{D}$, we can define the characteristic determinant of matrix $\mathcal{E}^{\beta,\eta}_{[a,b]}$ by
\begin{equation*}
\varphi^{\beta,\eta}_{[a,b]}(z):=\det(z-\mathcal{E}^{\beta,\eta}_{[a,b]}).
\end{equation*}

According to the results in \cite[Theorem 2]{Wang-JMAA}, the relation between this characteristic determinant and the $n$-step transfer matrix is
\begin{equation}\label{relation}
M_n(\omega,z;x)=(\sqrt{z})^{-n}\Big(\prod_{j=0}^{n-1}\frac{1}{\rho_j}\Big)
\left(
\begin{matrix}
z\varphi^{\beta,\eta}_{[1,n-1]} & \frac{z\varphi^{\beta,\eta}_{[1,n-1]}-\varphi^{\beta,\eta}_{[0,n-1]}}{\alpha_{-1}}\\
z\big(\frac{z\varphi^{\beta,\eta}_{[1,n-1]}-\varphi^{\beta,\eta}_{[0,n-1]}}{\alpha_{-1}}\big)^* &(\varphi^{\beta,\eta}_{[1,n-1]})^{*}
\end{matrix}
\right).
\end{equation}

\section{Basic tools}
\subsection{Large deviation estimates (LDT)}
\begin{theorem}\label{LDT-matrix}\cite[Theorem 3.2]{Zhang-Piao}
Let $\omega\in \mathbb{T}^{d}(p,q)$, $z\in \partial \mathbb{D}$, and suppose the Lyapunov exponent satisfies $L(\omega,z)>\gamma>0$. The there exist constants $\sigma=\sigma(p,q)$, $\tau=\tau(p,q)$ in $(0,1)$, and $C_{0}=C_{0}(p,q,h)$ such that for all $n\geq 1$, the following holds:
\begin{equation}\label{LDT1}
\mathrm{mes}\{x\in \mathbb{T}^{d}: |\log\|M_{n}(\omega,z;x)\|-nL_{n}(\omega,z)|> n^{1-\tau}\}<\exp(-C_{0}n^{\sigma}).
\end{equation}
\end{theorem}

\begin{theorem}\label{LDT-determinant}\cite[Remark 3.20]{Zhang-Piao}
Suppose $\omega\in \mathbb{T}^{d}(p,q)$, $z\in \partial \mathbb{D}$, and $L(\omega,z)>\gamma>0$. There exist $\tau=\tau(p,q)$, $\nu=\nu(p,q)$, $ \tau, \nu \in (0,1)$, $C=C(p,q,h)$ such that for $n\geq 1$ one has that
\begin{equation}\label{LDT2}
\mathrm{mes}\{x\in \mathbb{T}^{d}:|\log|\varphi^{\beta,\eta}_{[0,n-1]}(\omega,z;x)|-nL_{n}(\omega,z)|>n^{1-\tau}\}<\exp(-Cn^{\nu}).
\end{equation}
\end{theorem}

\subsection{Estimates for the Lyapunov exponent}
\begin{lemma}\label{Lemma3.3}\cite[Lemma 3.3]{Zhang-Piao}
Assume $\omega\in \mathbb{T}^{d}(p,q)$, $z\in \partial \mathbb{D}$, and $L(\omega,z)>\gamma>0$. Then for any $n\geq 2$,
\begin{equation*}
0\leq L_{n}(\omega,z)-L(\omega,z)<C\frac{(\log n)^{1/\sigma}}{n},
\end{equation*}
where $C=C(p,q,z,\gamma)$ and $\sigma$ is as in (LDT).
\end{lemma}

\begin{lemma}\label{Lemma3.4}\cite[Corollary 3.9]{Zhang-Piao}
Let $\omega\in \mathbb{T}^{d}$, $z\in \partial \mathbb{D}$, and $L(\omega,z)>\gamma>0$. There exists $C=C(z)$ such that
\begin{equation*}
|L_{n}(\omega,z;y)-L_{n}(\omega,z)|\leq C\sum_{i=1}^{d}|y_{i}|
\end{equation*}
for all $|y|<h$ uniformly in $n$. Particularly, the same bound holds with $L$ instead of $L_{n}$.
\end{lemma}

\begin{lemma}\label{Lemma3.5}\cite[Lemma 3.10]{Zhang-Piao}
Assume $\omega\in\mathbb{T}^{d}(p,q)$, $z\in \partial \mathbb{D}$, and $L(\omega,z)>\gamma>0$. Then for all $n\geq 1$,
\begin{equation}\label{Lem3.11-(1)}
\sup\limits_{x\in \mathbb{T}^{d}}\log\|M_{n}(\omega,z;x)\|\leq nL_{n}(\omega,z)+Cn^{1-\tau},
\end{equation}
where $C=C(p,q,z,\gamma)$ and $\tau$ as in (LDT).
\end{lemma}

\begin{lemma}\label{Lemma3.6}\cite[Corollary 3.11]{Zhang-Piao}
Suppose $\omega_{0}\in \mathbb{T}(p,q)$, $z_{0}\in \partial \mathbb{D}$ and $L(\omega_{0},z_{0})>\gamma>0$. Let $\sigma$, $\tau$ be as in (LDT). Then for all $n\geq N_{0}(p,q,z_{0},\gamma)$ and $(\omega,z,y)\in \mathbb{C}^{d}\times \partial \mathbb{D}\times \mathbb{R}^{d}$ such that
\begin{equation*}
|y|<\frac{1}{n}, \quad |\omega-\omega_{0}|,\quad |z-z_{0}|<\exp(-(\log n)^{8/\sigma}),
\end{equation*}
we have that
\begin{equation*}
\sup\limits_{x\in \mathbb{T}^{d}}\log\|M_{n}(\omega,z;x+iy)\|\leq nL_{n}(\omega_{0},z_{0})+Cn^{1-\tau},
\end{equation*}
where $C=C(p,q,z_{0},\gamma)$.
\end{lemma}

\begin{remark}\label{Remark3.7}
According to the relation between the transfer matrix and the characteristic determinant, we have that
$$\sup\limits_{x\in \mathbb{T}^{d}}\log|\varphi^{\beta,\eta}_{[0,n-1]}(x+iy)|\leq nL_{n}(\omega_{0},z_{0})+Cn^{1-\tau}.$$
\end{remark}
\subsection{Cartan's estimate}
In what follows, we denote $\mathcal{D}(z_{0},r)=\{z\in \mathbb{C}:|z-z_{0}|<r\}$.
\begin{definition 1}\cite[Definition 2.12]{GS08-GAFA}
Let $H\geq 1$. For an arbitrary set $\mathcal{B}\subset \mathcal{D}(z_{0},1)\subset \mathbb{C}$ we say that $\mathcal{B}\in \mathrm{Car}_{1}(H,K)$ if $\mathcal{B}\subset \mathop{\cup}\limits^{j_{0}}_{j=1}\mathcal{D}(z_{j},r_{j})$ with $j_{0}\leq K$, and
\begin{equation}\label{def2-(1)}
\sum_{j}r_{j}<e^{-H}.
\end{equation}
If $d\geq 2$ is an integer and $\mathcal{B}\subset\mathop{\prod}\limits_{j=1}^{d}\mathcal{D}(z_{j,0},1)\subset \mathbb{C}^{d}$, then we define inductively that $\mathcal{B}\in \mathrm{Car}_{d}(H,K)$ if for any $1\leq j\leq d$ there exists $\mathcal{B}_{j}\subset\mathcal{D}(z_{j,0},1)\subset \mathbb{C}$, $\mathcal{B}_{j}\in \mathrm{Car}_{1}(H,K)$ so that $\mathcal{B}^{(j)}_{z}\in\mathrm{Car}_{d-1}(H,K)$ for any $z\in \mathbb{C}\backslash\mathcal{B}_{j}$, here $\mathcal{B}^{(j)}_{z}=\{(z_{1},\ldots,z_{d})\in \mathcal{B}:z_{j}=z\}$.
\end{definition 1}

\begin{lemma}\label{Lemma3.23}\cite[Lemma 2.15]{GS08-GAFA}
Let $\varphi(z_{1},\ldots,z_{d})$ be an analytic function defined on a polydisk $\mathcal{P}=\mathop{\prod}\limits_{j=1}^{d}\mathcal{D}(z_{j,0},1)$, $z_{j,0}\in \mathbb{C}$. Let $M\geq\sup\limits_{z\in\mathcal{P}}\log|\varphi(z)|$, $m\leq\log|\varphi(z_{0})|$, $z_{0}=(z_{1,0},\ldots,z_{d,0})$. Given $H\gg 1$, there exists a set $\mathcal{B}\subset \mathcal{P}$, $\mathcal{B}\in \mathrm{Car}_{d}(H^{1/d},K)$, $K=C_{d}H(M-m)$, such that
\begin{equation}\label{Lem3.23-(1)}
\log|\varphi(z)|>M-C_{d}H(M-m)
\end{equation}
for any $z\in \frac{1}{6}\mathcal{P}\backslash \mathcal{B}$. Furthermore, when $d=1$ we can take $K=C(M-m)$ and keep only the disks of $\mathcal{B}$ containing a zero of $\varphi$ in them.
\end{lemma}

\begin{lemma}\label{Cartan-measure}\cite[Lemma 2.11]{GSV16-arXiv}
If $\mathcal{B}\in \mathrm{Car}_{d}(H,K)$ then
$$\mathrm{mes}_{\mathbb{C}^{d}}(\mathcal{B})\leq C(d)\exp(-H)\quad and \quad \mathrm{mes}_{\mathbb{R}^{d}}(\mathcal{B}\cap\mathbb{R}^{d})\leq C(d)\exp(-H).$$
\end{lemma}

The following statement is a consequence of Cartan's estimate, we refer to as the spectral form of (LDT).
\begin{lemma}\label{spectrl-form-(LDT)}\cite[Lemma 3.22]{Zhang-Piao}
Assume $x\in \mathbb{T}^{d}$, $\omega\in \mathbb{T}^{d}(p,q)$, $z\in\partial \mathbb{D}$ and $L(\omega,z)>\gamma>0$. Let $\tau$, $\nu$ as in (LDT). If $n\geq N(p,q,z,\gamma)$ and
\begin{equation*}
\|(\mathcal{E}^{\beta,\eta}_{[0,n-1]}-z)^{-1}\|\leq C\exp(n^{\nu/2}),
\end{equation*}
then
\begin{equation*}
\log|\varphi^{\beta,\eta}_{[0,n-1]}(x)|>nL_{n}(\omega,z)-n^{1-\tau/2}.
\end{equation*}
\end{lemma}

\subsection{Poisson formula}
For $z\in \mathbb{C}$, $\beta, \eta\in\partial \mathbb{D}$, we can define the polynomial
\begin{equation*}
\phi^{\beta,\eta}_{[a,b]}(z):=(\rho_{a}\cdots\rho_{b})^{-1}\varphi^{\beta,\eta}_{[a,b]}(z).
\end{equation*}
Note that when $a>b$, $\phi^{\beta,\eta}_{[a,b]}(z)=1$.

Since the equation $\mathcal{E}u=zu$ is equivalent to $(z\mathcal{L}^{*}-\mathcal{M})u=0$, the associated finite-volume Green's functions are as follows:
\begin{equation*}
G^{\beta,\eta}_{[a,b]}(z)=\big(z(\mathcal{L}^{\beta,\eta}_{[a,b]})^{*}-\mathcal{M}^{\beta,\eta}_{[a,b]}\big)^{-1},
\end{equation*}
\begin{equation*}
G^{\beta,\eta}_{[a,b]}(j,k;z)=\langle \delta_{j},G^{\beta,\eta}_{[a,b]}(z)\delta_{k}\rangle,\quad j,k\in [a,b].
\end{equation*}
According to \cite[Proposition 3.8]{Kruger13-IMRN} and \cite[Section B.1]{Zhu-JAT}, for $\beta, \eta\in\partial \mathbb{D}$, the Green's function has the expression
\begin{equation*}
|G^{\beta,\eta}_{[a,b]}(j,k;z)|=\frac{1}{\rho_{k}}\Bigg| \frac{\phi^{\beta,\cdot}_{[a,j-1]}(z)\phi^{\cdot,\eta}_{[k+1,b]}(z)}{\phi^{\beta,\eta}_{[a,b]}(z)}\Bigg|,\quad a\leq j\leq k\leq b,
\end{equation*}
where ``$\cdot$" stands for the unchanged Verblunsky coefficient.

From \cite[Lemma 3.9]{Kruger13-IMRN}, if $u$ satisfies $\tilde{\mathcal{E}}u=zu$, Poisson's formula reads
\begin{align*}
u(m)=&G^{\beta,\eta}_{[a,b]}(a, m;z)
\begin{cases}
(z\overline{\beta}-\alpha_{a})u(a)-\rho_{a}u(a+1), \quad & a\text{ even,}\\
(z\alpha_{a}-\beta)u(a)+z\rho_{a}u(a+1),& a \text { odd,}
\end{cases}\\
&+G^{\beta,\eta}_{[a,b]}(m,b;z)
\begin{cases}
(z\overline{\eta}-\alpha_{b})u(b)-\rho_{b}u(b-1), \quad & b\text{ even,}\\
(z\alpha_{b}-\eta)u(b)+z\rho_{b-1}u(b-1),& b \text { odd}
\end{cases}
\end{align*}
for $a<m<b$.
\begin{lemma}\label{covering-lemma}\cite[Lemma 3.23]{Zhang-Piao}
Let $x, \omega\in \mathbb{T}^{d}$, $z\in \partial \mathbb{D}$ and $(a,b)\subset \mathbb{Z}$. If for any $m\in [a+1,b-1]$, there exists an interval $I_{m}=[a_{m},b_{m}]\subset[a+1,b-1]$ containing $m$ such that
\begin{small}
\begin{equation*}
|G^{\beta,\eta}_{I_{m}}(a_{m},m;z)|
\begin{cases}
|z\overline{\beta}-\alpha_{a}|+\rho_{a}, & a\text{ even,}\\
|z\alpha_{a}-\beta|+\rho_{a},& a \text { odd,}
\end{cases}
+|G^{\beta,\eta}_{I_{m}}(m,b_{m};z)|
\begin{cases}
|z\overline{\eta}-\alpha_{b}|+\rho_{b},  &b\text{ even}\\
|z\alpha_{b}-\eta|+\rho_{b-1},& b \text { odd}
\end{cases}
<1,
\end{equation*}
\end{small}
then $z\notin \sigma(\mathcal{E}^{\beta,\eta}_{[a,b]})$.
\end{lemma}

\begin{lemma}\label{Green-function-estimate}\cite[Lemma 3.25]{Zhang-Piao}
Suppose that $x_{0}\in \mathbb{T}^{d}$, $\omega_{0}\in \mathbb{T}^{d}(p,q)$, $z_{0}\in \partial \mathbb{D}$ and $L(\omega_{0},z_{0})>\gamma>0$. Let $K\in \mathbb{R}$ and $\tau$ be as in (LDT). There exists $C_{0}=C_{0}(p,q,z_{0},\gamma)$ such that if $n\geq N(p,q,z_{0},\gamma)$ and
\begin{equation}\label{Lem3.28-(1)}
\log|\varphi^{\beta,\eta}_{[0,n-1]}(\omega_{0},z_{0};x_{0})|>nL_{n}(\omega_{0},z_{0})-K,
\end{equation}
then for any $(\omega,z,x)\in \mathbb{T}^{d}\times\partial \mathbb{D}\times\mathbb{T}^{d}$ with $|x-x_{0}|, |\omega-\omega_{0}|, |z-z_{0}|<\exp(-(K+C_{0}n^{1-\tau}))$ we have that
\begin{equation}\label{Lem3.28-(2)}
|G^{\beta,\eta}_{[0,n-1]}(j,k;z)|\leq \exp(-\frac{\gamma}{2}|k-j|+K+2C_{0}n^{1-\tau})
\end{equation}
and
\begin{equation}\label{Lem3.28-(3)}
\|G^{\beta,\eta}_{[0,n-1]}(z)\|\leq \exp(K+3C_{0}n^{1-\tau}).
\end{equation}
\end{lemma}

We refer to the next statement as the covering form of (LDT).
\begin{lemma}\label{covering-form-(LDT)}\cite[Lemma 3.26]{Zhang-Piao}
Suppose that $n\gg 1$, $x_{0}\in \mathbb{T}^{d}$, $\omega_{0}\in \mathbb{T}^{d}(p,q) $, $z_{0}\in \partial \mathbb{D}$ and $L(\omega_{0},z_{0})>\gamma>0$. Let $\tau, \nu$ be as in (LDT). Suppose that for each point $m\in [0,n-1]$, there exists an interval $I_{m}\subset [0,n-1]$ such that: \\
$\mathrm{(i)}$ $\mathrm{dist}(m, [0,n-1]\backslash I_{m})\geq |I_{m}|/100$,\\
$\mathrm{(ii)}$ $|I_{m}|\geq C(p,q,z_{0},\gamma)$,\\
$\mathrm{(iii)}$ $\log|\varphi^{\beta,\eta}_{I_{m}}(\omega_{0},z_{0};x_{0})|>|I_{m}|L_{|I_{m}|}(\omega_{0},z_{0})-|I_{m}|^{1-\tau/4}$.\\
Then for any $(\omega,z,x)\in \mathbb{T}^{d}\times \partial \mathbb{D}\times\mathbb{T}^{d} $ such that
$$|x-x_{0}|,\, |\omega-\omega_{0}|,\, |z-z_{0}|<\exp(-2\max\limits_{m}|I_{m}|^{1-\tau/4}),$$
we have that
$$\mathrm{dist}(z,\sigma(\mathcal{E}^{\beta,\eta}_{[0,n-1]}))\geq \exp (-2\max\limits_{m}|I_{m}|^{1-\tau/4}).$$
In addition, if $\omega\in \mathbb{T}^{d}(p,q)$ and $\max\limits_{m}|I_{m}|\leq n^{\nu/2}$, then
$$\log|\varphi^{\beta,\eta}_{[0,n-1]}|>nL_{n}(\omega,z)-n^{1-\tau/2}.$$
\end{lemma}

Now we give another formulation of the covering form of (LDT), which is more suitable for the setting of this paper.
\begin{lemma}\label{another-covering-form-LDT}
Assume $x_{0}\in \mathbb{T}^{d}$, $\mathcal{Z}\subset \partial \mathbb{D}$, and $L(z)>\gamma>0$ for $z\in \mathcal{Z}$. Let $\nu$ be as in (LDT) and $a<b$ be integers. Suppose that for each point $m\in [a,b]$ there exists an interval $J_{m}$ such that $m\in J_{m}$ and:\\
$\mathrm{(1)}$ $\mathrm{dist}(m,\partial J_{m})\geq |J_{m}|/100$,\\
$\mathrm{(2)}$ $\mathrm{dist}(\sigma(\mathcal{E}^{\beta,\eta}_{J_{m}}(x_{0})),\mathcal{Z})\geq \exp(-K)$, with $K<\frac{1}{2}\mathop{\min}\limits_{m}|J_{m}|^{\nu/2}$,\\
$\mathrm{(3)}$ $K\geq C(p,q,z,\gamma)$.\\
Let $J=\mathop{\cup}\limits_{m\in [a,b]}J_{m}$. Then for any $|x-x_{0}|<\exp(-2K)$ we have that
\begin{equation*}
\mathrm{dist}(\sigma(\mathcal{E}^{\beta,\eta}_{J}(x)),\mathcal{Z})\geq \frac{1}{2}\exp(-K).
\end{equation*}
\end{lemma}
\begin{proof}
We only need to consider the case $\mathcal{Z}=\{z_{0}\}$. For the general case, we can apply this particular case to each $z_{0}\in \mathcal{Z}$.

Now we set up some intervals for which we will be able to apply the covering lemma, i.e. Lemma \ref{covering-form-(LDT)}. Let $J_{m}=[c_{m},d_{m}]$. Then we have that
\begin{equation*}
J=[c,d],\quad c=\inf_{m}c_{m},\quad d=\sup_{m}d_{m}.
\end{equation*}
Let $m_{-}=\sup\{m\in[a,b]: c_{m}=c\}$, $m_{+}=\inf\{m\in[a,b]: d_{m}=d\}$,
\begin{equation*}
I_{m}=
\begin{cases}
J_{m_{-}}, & m\in[c,m_{-}]\\
J_{m},& m\in[m_{-},m_{+}]\\
J_{m_{+}}, &m\in [m_{+},d]
\end{cases}.
\end{equation*}
Then $\mathrm{dist}(m, J\backslash I_{m})\geq |I_{m}|/100$.

Take $m\in [c,d]$. Based on (2), for any
\begin{equation*}
|x-x_{0}|<\exp(-2K),\quad |z-z_{0}|\leq\frac{1}{2}\exp(-K),
\end{equation*}
we have that
\begin{align*}
\mathrm{dist}(\sigma(\mathcal{E}^{\beta,\eta}_{I_{m}}(x)),z)&\geq\mathrm{dist}(\sigma(\mathcal{E}^{\beta,\eta}_{I_{m}}(x)),z_{0})-\mathrm{dist}(z,z_{0})\\
&\geq\exp(-K)-\frac{1}{2}\exp(-K)\\
&>\frac{1}{4}\exp(-K)>\exp(-|I_{m}|^{\nu/2}).
\end{align*}
According to the spectral form of (LDT), we have that
\begin{equation*}
\log|\varphi^{\beta,\eta}_{I_{m}}(x)|>|I_{m}|L_{|I_{m}|}(\omega,z)-|I_{m}|^{1-\tau/2}.
\end{equation*}
Applying Lemma \ref{Green-function-estimate}, one can obtain that
\begin{equation*}
|G^{\beta,\eta}_{I_{m}}(m,k;z)|\leq \exp(-\frac{\gamma}{2}|m-k|+\frac{3}{2}|I_{m}|^{1-\tau/2}).
\end{equation*}
By Lemma \ref{covering-lemma}, we know that $z\notin \sigma(\mathcal{E}_{J}^{\beta,\eta}(x))$ for any $|z-z_{0}|\leq\frac{1}{2}\exp(-K)$. Therefore,
\begin{equation*}
\mathrm{dist}(\sigma(\mathcal{E}^{\beta,\eta}_{J}(x)),\mathcal{Z})\geq \frac{1}{2}\exp(-K).
\end{equation*}
\end{proof}

\subsection{Finite scale localization}
\begin{lemma}\label{another-finite-scale-localization}
Let $x_{0}\in \mathbb{T}^{d}$, $z_{0}\in \partial \mathbb{D}$, and $L(z_{0})>\gamma>0$. Let $\nu$ be as in (LDT) and $0<\hat{\beta}<\nu/2$. Let $n\geq N_{0}$ be integers. Suppose that for any $3N_{0}/2<|m|\leq n$ there exists an interval $J_{m}$ such that $m\in J_{m}$, $\mathrm{dist}(m,\partial J_{m})\geq N_{0}-N_{0}^{\frac{1}{2}}$, $|J_{m}|\leq 10 N_{0}$, and
\begin{equation*}
\mathrm{dist}(\sigma(\mathcal{E}_{J_{m}}^{\beta,\eta}(x_{0})),z_{0})\geq \exp(-N_{0}^{\hat{\beta}}).
\end{equation*}
Let $[-n',n'']=[-3N_{0}/2,3N_{0}/2]\cup \mathop{\cup}\limits_{3N_{0}/2<|m|\leq n}J_{m}$. Then the following holds for $N_{0}\geq C(p,q,\hat{\beta})$. If
\begin{equation*}
|x-x_{0}|<\exp(-2N_{0}^{\hat{\beta}}), \quad \big|z_{k}^{[-n',n'']}(\omega,z;x)-z_{0}\big|<\frac{1}{4}\exp(-N_{0}^{\hat{\beta}}),
\end{equation*}
then
$$\big|u_{k}^{[-n',n'']}(\omega,x;s)\big|<\exp(-\frac{\gamma|s|}{20}), \,\, |s|\geq 3N_{0}/4.$$
\end{lemma}
\begin{proof}
Take $x, z=z_{k}^{[-n',n'']}(\omega,z;x)$ satisfying the assumptions. Without loss of generality, we assume that $s\geq 3N_{0}/4$. Let $\hat{d}=\lceil s-N_{0}/2\rceil$. Then $\hat{d}>s/3$. Let
$$J=\cup\{J_{m}: m\in[s-\hat{d},s+\hat{d}+N_{0}]\cap(3N_{0}/2,n]\}.$$
Due to $m\in J_{m}$, $\mathrm{dist}(m,\partial J_{m})\geq N_{0}-N_{0}^{\frac{1}{2}}$, we have that $m+[-(N_{0}-N_{0}^{\frac{1}{2}}),N_{0}-N_{0}^{\frac{1}{2}}]\subseteq J_{m}$, $N_{0}<|J|\lesssim \hat{d}$, $s\in J$, and $\mathrm{dist}(s,[-n',n'']\backslash J)\geq \hat{d}$.

According to Lemma \ref{another-covering-form-LDT}, one can obtain that
$$\mathrm{dist}(\sigma(\mathcal{E}_{J}^{\beta,\eta}(x)),z)\geq\frac{1}{4}\exp(-N_{0}^{\hat{\beta}})>\exp(-|J|^{\nu/2}).$$
Moreover, with the aid of the spectral form of (LDT), we have that
$$\log|\varphi_{J}^{\beta,\eta}(x)|>|J|L_{|J|}(\omega,z)-|J|^{1-\tau/2}.$$
Let $J=[j_{1},j_{2}]$. According to Poisson's formula,
\begin{align*}
u_{k}^{[-n',n'']}(s)=&G^{\beta,\eta}_{J}(j_{1}, s;z)
\begin{cases}
(z\overline{\beta}-\alpha_{j_{1}})u(j_{1})-\rho_{j_{1}}u(j_{1}+1), \quad & j_{1}\text{ even,}\\
(z\alpha_{j_{1}}-\beta)u(j_{1})+z\rho_{j_{1}}u(j_{1}+1),& j_{1} \text { odd,}
\end{cases}\\
&+G^{\beta,\eta}_{J}(s,j_{2};z)
\begin{cases}
(z\overline{\eta}-\alpha_{j_{2}})u(j_{2})-\rho_{j_{2}}u(j_{2}-1), \quad & j_{2}\text{ even,}\\
(z\alpha_{j_{2}}-\eta)u(j_{2})+z\rho_{j_{2}-1}u(j_{2}-1),& j_{2} \text { odd}
\end{cases}
\end{align*}
for $j_{1}<s<j_{2}$. Combining with Lemma \ref{Green-function-estimate}, we can derive that
\begin{equation*}
|u_{k}^{[-n',n'']}(s)|\leq 8\exp(-\frac{\gamma}{2}\hat{d}+C|J|^{1-\tau/2})<\exp(-\frac{\gamma}{6}\hat{d})<\exp(-\frac{\gamma}{20}s)
\end{equation*}
(recall that $u$ is normalized).
\end{proof}

In what follows, when we deal with the difference between two vectors with different dimensions, we make the dimension of a low dimensional vector equal to the dimension of a high dimensional vector by adding zero element. For example, if \begin{equation*}
\xi=(x_{-5},x_{-4},\ldots,x_{-1},x_{0},x_{1},\ldots,x_{4},x_{5}),\,\,\upsilon=(y_{-7},y_{-6},\ldots,y_{-1},y_{0},y_{1},\ldots,y_{6},y_{7})
\end{equation*}
and $n>m$, then $$\upsilon-\xi=(y_{-7},y_{-6},y_{-5}-x_{-5},\ldots,y_{-1}-x_{-1},y_{0}-x_{0},y_{1}-x_{1},\ldots,y_{5}-x_{5},y_{6},y_{7}).$$
When the columns of a matrix are greater than the rows of a vector, we deal with the product of them in the same way.

\begin{lemma}\label{finite-scale-lemma-2}
We use the notation and assumptions of Lemma \ref{another-finite-scale-localization}. We further assume that there exist integers $|N_{0}'-N_{0}|<N_{0}^{\frac{1}{2}}$, $|N_{0}''-N_{0}|<N_{0}^{\frac{1}{2}}$, and $k_{0}$, such that the following conditions hold:\\

$\mathrm{(i)}$ $|z_{k_{0}}^{[-N_{0}',N_{0}'']}(\omega,z;x_{0})-z_{0}|<\exp(-2N_{0}^{\hat{\beta}})$,\\

$\mathrm{(ii)}$ $|u_{k_{0}}^{[-N_{0}',N_{0}'']}(\omega,x_{0};-(N_{0}'-i))|, |u_{k_{0}}^{[-N_{0}',N_{0}'']}(\omega,x_{0};N_{0}''-i)|<\exp(-2N_{0}^{\hat{\beta}})$, $i=0,1,2,3$.\\
Then there exist $z_{k}^{[-n',n'']}$, $u_{k}^{[-n',n'']}$ with $[-N_{0}',N_{0}'']\subset [-n',n'']$, such that the following estimates hold for any $|x-x_{0}|<\exp(-2N_{0}^{\hat{\beta}})$, provided $N_{0}\geq C(p,q,\hat{\beta})$:\\

$\mathrm{(1)}$ $\big|z_{k}^{[-n',n'']}(\omega,z;x)-z_{k_{0}}^{[-N_{0}',N_{0}'']}(\omega,z;x)\big|<\exp(-\gamma N_{0}/40)$;\\

$\mathrm{(2)}$ $\big|z_{j}^{[-n',n'']}(\omega,z;x)-z_{k}^{[-n',n'']}(\omega,z;x)\big|>\frac{1}{8}\exp(-N_{0}^{\hat{\beta}})$, $j\neq k$;\\

$\mathrm{(3)}$ $\big|u_{k}^{[-n',n'']}(\omega,x;s)\big|<\exp(-\frac{\gamma |s|}{20})$, $|s|\geq 3N_{0}/4$;\\

$\mathrm{(4)}$ $\big\|u_{k}^{[-n',n'']}(\omega,x;\cdot)-u_{k_{0}}^{[-N_{0}',N_{0}'']}(\omega,x;\cdot)\big\|<\exp(-\gamma N_{0}/40)$.
\end{lemma}
\begin{proof}
Due to the condition (ii),
$$\big\|(\mathcal{E}_{[-n',n'']}^{\beta,\eta}(\omega,x_{0})-z_{k_{0}}^{[-N_{0}',N_{0}'']}(\omega,x_{0}))u_{k_{0}}^{[-N_{0}',N_{0}'']}(\omega,x_{0})\big\|
<8\exp(-2N_{0}^{\hat{\beta}}).$$
Applying part (a) of Corollary \ref{eigenvector}, one can obtain that there exists $z_{k}^{[-n',n'']}(\omega,x_{0})$ such that
$$\big|z_{k}^{[-n',n'']}(\omega,x_{0})-z_{k_{0}}^{[-N_{0}',N_{0}'']}(\omega,x_{0})\big|<8\sqrt{2}\exp(-2N_{0}^{\hat{\beta}}).$$
Then for $|x-x_{0}|<\exp(-2N_{0}^{\hat{\beta}})$, we have that
\begin{align*}
&\big|z_{k}^{[-n',n'']}(\omega,x)-z_{k_{0}}^{[-N_{0}',N_{0}'']}(\omega,x)\big|\\
&\leq \big|z_{k}^{[-n',n'']}(\omega,x)-z_{k}^{[-n',n'']}(\omega,x_{0})\big|+\big|z_{k}^{[-n',n'']}(\omega,x_{0})-z_{k_{0}}^{[-N_{0}',N_{0}'']}(\omega,x_{0})\big|\\
&\quad +\big|z_{k_{0}}^{[-N_{0}',N_{0}'']}(\omega,x_{0})-z_{k_{0}}^{[-N_{0}',N_{0}'']}(\omega,x)\big|\\
&\ll \exp(-N_{0}^{\hat{\beta}})
\end{align*}
and
$$\big|z_{k}^{[-n',n'']}(\omega,x)-z_{0}\big|\leq\big|z_{k}^{[-n',n'']}(\omega,x)-z_{k_{0}}^{[-N_{0}',N_{0}'']}(\omega,x)\big|
+\big|z_{k_{0}}^{[-N_{0}',N_{0}'']}(\omega,x)-z_{0}\big|\ll\exp(-N_{0}^{\hat{\beta}}).$$
Therefore, the statement (3) follows with the aid of Lemma \ref{another-finite-scale-localization}.

Then one can obtain that
\begin{align*}
&\big\|(\mathcal{E}_{[-N_{0}',N_{0}'']}^{\beta,\eta}(\omega,x)-z_{k}^{[-n',n'']}(\omega,x))u_{k}^{[-n',n'']}(\omega,x;\cdot)\big\|\\
&=\big\|(\mathcal{E}_{[-N_{0}',N_{0}'']}^{\beta,\eta}(\omega,x)-\mathcal{E}_{[-n',n'']}^{\beta,\eta}(\omega,x))u_{k}^{[-n',n'']}(\omega,x;\cdot)\big\|\\
&\lesssim \exp(-\gamma (N_{0}-N_{0}^{\frac{1}{2}})/20).
\end{align*}
According to the part (b) of Corollary \ref{eigenvector}, with $\eta=c\exp(-N_{0}^{\hat{\beta}})$, $c\ll 1$, statements (1) and (4) hold. Now we prove the statement (2). Assume that there exist $j\neq k$ and $x$ such that
$$\big|z_{j}^{[-n',n'']}(\omega,z;x)-z_{k}^{[-n',n'']}(\omega,z;x)\big|\leq\frac{1}{8}\exp(-N_{0}^{\hat{\beta}}).$$
Thus, we get that
\begin{align*}
&\big|z_{j}^{[-n',n'']}(\omega,z;x)-z_{k_{0}}^{[-N_{0}',N_{0}'']}(\omega,z;x)\big|\\
&\leq\big|z_{j}^{[-n',n'']}(\omega,z;x)-z_{k}^{[-n',n'']}(\omega,z;x)\big|+\big|z_{k}^{[-n',n'']}(\omega,z;x)-z_{k_{0}}^{[-N_{0}',N_{0}'']}(\omega,z;x)\big|\\
&\leq\frac{1}{8}\exp(-N_{0}^{\hat{\beta}})+\exp(-\gamma N_{0}/40)<\frac{1}{4}\exp(-N_{0}^{\hat{\beta}}).
\end{align*}
As a consequence, we have that
\begin{align*}
\big|z_{j}^{[-n',n'']}(\omega,z;x)-z_{0}\big|&\leq \big|z_{j}^{[-n',n'']}(\omega,z;x)-z_{k_{0}}^{[-N_{0}',N_{0}'']}(\omega,z;x)\big|+\big|z_{k_{0}}^{[-N_{0}',N_{0}'']}(\omega,z;x)-z_{0}\big|\\
&\leq\frac{1}{8}\exp(-N_{0}^{\hat{\beta}})+\exp(-\gamma N_{0}/40)+\exp(-2N_{0}^{\hat{\beta}})<\frac{1}{4}\exp(-N_{0}^{\hat{\beta}}).
\end{align*}
Using Lemma \ref{another-finite-scale-localization}, it follows that
$$\big|u_{j}^{[-n',n'']}(\omega,x;s)\big|<\exp(-\frac{\gamma|s|}{20}),\quad |s|\geq 3N_{0}/4.$$
Just as above we can obtain that
$$\big\|u_{j}^{[-n',n'']}(\omega,x;\cdot)-u_{k_{0}}^{[-N_{0}',N_{0}'']}(\omega,x;\cdot)\big\|<\exp(-\gamma N_{0}/40).$$
Thus,
\begin{align*}
&\big\|u_{j}^{[-n',n'']}(\omega,x;\cdot)-u_{k}^{[-n',n'']}(\omega,x;\cdot)\big\|\\
&\leq\big\|u_{j}^{[-n',n'']}(\omega,x;\cdot)-u_{k_{0}}^{[-N_{0}',N_{0}'']}(\omega,x;\cdot)\big\|+\big\|u_{k_{0}}^{[-N_{0}',N_{0}'']}(\omega,x;\cdot)-u_{k}^{[-n',n'']}(\omega,x;\cdot)\big\|\\
&<2\exp(-\gamma N_{0}/40)<1.
\end{align*}
On the other hand, $u_{j}^{[-n',n'']}(\omega,x;\cdot)$, $u_{k}^{[-n',n'']}(\omega,x;\cdot)$ are normalized eigenvectors with different eigenvalues and
$$\big\|u_{j}^{[-n',n'']}(\omega,x;\cdot)-u_{k}^{[-n',n'']}(\omega,x;\cdot)\big\|^{2}=2.$$
This makes a contradiction. Then the statement (2) holds.
\end{proof}

\subsection{Semialgebraic sets}
In this section, semialgebraic sets will be introduced by approximating the Verblunsky coefficients $\alpha_{i}$ with a polynomial $\tilde{\alpha}_{i}$. More precisely, given $n\geq 1$, by truncating $\alpha_{i}$'s Fourier series and the Taylor series of the trigonometric functions, one can obtain a polynomial $\tilde{\alpha}_{i}$ of degree less than $ n^{4}$ such that
\begin{equation}\label{VCA}
\|\alpha_{i}-\tilde{\alpha}_{i}\|_{\infty}\lesssim \exp(-n^{2}).
\end{equation}
Let $\tilde{\mathcal{E}}_{[0,n-1]}^{\beta,\eta}$ be the matrix with this truncated Verblunsky coefficient and the refined boundary condition $\beta, \eta\in \partial\mathbb{D}$. Let $z_{j}^{[0,n-1]}(x,\omega)$ and $\tilde{z}_{j}^{[0,n-1]}(x,\omega)$ be the eigenvalues of $\mathcal{E}_{[0,n-1]}^{\beta,\eta}$ and $\tilde{\mathcal{E}}_{[0,n-1]}^{\beta,\eta}$, respectively. Then we have that
\begin{equation}\label{eigenvalue-estimate}
|z_{j}^{[0,n-1]}(x,\omega)-\tilde{z}_{j}^{[0,n-1]}(x,\omega)|\leq \|\mathcal{E}_{[0,n-1]}^{\beta,\eta}(x,\omega)-\tilde{\mathcal{E}}_{[0,n-1]}^{\beta,\eta}(x,\omega)\|\lesssim 8\|\alpha_{i}-\tilde{\alpha}_{i}\|_{\infty}.
\end{equation}


\begin{definition 2}\cite[Definition 9.1]{Bourgain-book}\label{semialgebraic-set}
A set $\mathcal{S}\subset \mathbb{R}^{n}$ is called semialgebraic if it is a finite union of sets defined by a finite number of polynomial equalities and inequalities. More precisely, let $\mathcal{P}=\{P_{1},\ldots,P_{s}\}\subset R[X_{1},\ldots,X_{n}]$ be a family of real polynomials whose degrees are bounded by $d$. A (closed) semialgebraic set $\mathcal{S}$ is given by an expression
\begin{equation}\label{semialgebraic(1)}
\mathcal{S}=\bigcup\limits_{j}\bigcap\limits_{l\in \mathcal{L}_{j}}\{R^{n}|P_{l}s_{jl}0\},
\end{equation}
where $\mathcal{L}_{j}\subset\{1,\ldots,s\}$ and $s_{jl}\in \{\geq,\leq,=\}$ are arbitrary. We say that $\mathcal{S}$ has degree at most $sd$, and its degree is the infimum of $sd$ over all representations as in \eqref{semialgebraic(1)}.
\end{definition 2}

\section{Inductive scheme for the bulk of the spectrum}
Let $\nu\ll\tau\ll 1$ be as in (LDT). Set $\hat{\delta}=(\nu')^{C_{0}}$, $\hat{\beta}=(\nu')^{C_{1}}$, $\hat{\mu}=(\nu')^{C_{2}}$ with $0<\nu'\leq \nu$ and $C_{0}, C_{1}, C_{2}>1$, satisfying the relations:
$$C_{1}+1<C_{2}<C_{0}<2C_{1}.$$
This implies that
\begin{equation}\label{exponent-relation}
\hat{\beta}^{2}\ll\hat{\delta}\ll\hat{\mu}\ll\hat{\beta}\nu\ll\hat{\beta}\ll\nu
\end{equation}
with the constants implied by ``$\ll$" being as large as we wish. Here $\nu'\leq c(C_{0},C_{1},C_{2})\nu$ is small enough.

Given an integer $s\geq 0$, let
$$z_{s}\in \partial\mathbb{D},\,\,N_{s}\in \mathbb{N},\,\,r_{s}:=\exp(-N_{s}^{\hat{\delta}}).$$
The inductive conditions are as follows:\\
(A) There exists integers $|N_{s}'-N_{s}|, |N_{s}''-N_{s}|<N_{s}^{\frac{1}{2}}$, a map $x_{s}: \Pi_{s}\rightarrow \mathbb{R}^{d}$,
$$\Pi_{s}=\mathcal{I}_{s}\times(\mathcal{D}(z_{s},r_{s})\cap\partial \mathbb{D}),\,\,\mathcal{I}_{s}=\phi_{s}+(-r_{s},r_{s})^{d-1},$$
with $\phi_{s}\in\mathbb{R}^{d-1}$, and $k_{s}$ such that for any $(\phi,z)\in \Pi_{s}$ we have that
\begin{align}
&z_{k_{s}}^{[-N_{s}',N_{s}'']}(x_{s}(\phi,z))=z,\label{(A)-(1)}\\
&\big|z_{j}^{[-N_{s}',N_{s}'']}(x_{s}(\phi,z))-z\big|>\exp(-N_{s}^{\hat{\delta}}),\,\,j\neq k_{s},\label{(A)-(2)}
\end{align}
and $x_{s}(\Pi_{s})\subset \mathbb{T}_{h/2}^{d}$.

To simplify notation we suppress $k_{s}$ and use $z^{[-N_{s}',N_{s}'']}$, $u^{[-N_{s}',N_{s}'']}$ instead.\\
(B) For each $(\phi,z)\in \Pi_{s}$,
\begin{equation}\label{(B)}
\big|u^{[-N_{s}',N_{s}'']}(x_{s}(\phi,z),n)\big|\leq\exp(-\gamma|n|/10),\,\,|n|\geq N_{s}/4.
\end{equation}
(C) Define
\begin{equation}\label{(C)}
\Upsilon_{s}=\{n\omega: 0\leq |n|\leq 3N_{s}/2\}.
\end{equation}
Take an arbitrary $\hat{h}\in \mathbb{T}^{d}$ with $\mathrm{dist}(\hat{h},\Upsilon_{s})\geq \exp(-N_{s}^{\hat{\mu}})$. Then for any $z\in \mathcal{D}(z_{s},r_{s})\cap \partial \mathbb{D}$, we get that
\begin{equation*}
\mathrm{mes}\Big\{\phi\in \mathcal{I}_{s}: \mathop{\max}\limits_{|n'|,|n''|<N_{s}^{\frac{1}{2}}}\mathrm{dist}(\sigma(\mathcal{E}_{[-N_{s}+n',N_{s}+n'']}^{\beta,\eta}(x_{s}(\phi,z)+\hat{h})),z)
<\exp(-N_{s}^{\hat{\beta}}/2)\Big\}<\exp(-N_{s}^{2\hat{\delta}}).
\end{equation*}
(D) Take an arbitrary unit vector $h_{0}\in\mathbb{C}^{d}$. Then for any $z\in \mathcal{D}(z_{s},r_{s})\cap \partial \mathbb{D}$, we have that
\begin{equation*}
\mathrm{mes}\{\phi\in \mathcal{I}_{s}: \log|\langle\nabla z^{[-N_{s}',N_{s}'']}(x_{s}(\phi,z)),h_{0}\rangle|<-N_{s}^{\hat{\mu}}/2\}<\exp(-N_{s}^{2\hat{\delta}}).
\end{equation*}

\begin{theorem}\label{bulk-theorem}
Assume the notation of the inductive conditions. Let $z_{0}\in\partial{\mathbb{D}}$, $L(z)>\gamma>0$ for $z\in \mathcal{D}(z_{0},r_{0})\cap \partial{\mathbb{D}}$. Let $N_{0}\geq 1$, $N_{s}=\lfloor N_{s-1}^{\hat{A}}\rfloor$, $\hat{A}=\hat{\beta}^{-1}$, $s\geq 1$. If $N_{0}\geq C(p,q,\hat{\beta})$ and conditions (A)-(D) hold with $s=0$, then for any $s\geq 1$ and $z_{s}\in \mathcal{D}(z_{s-1},r_{s-1})\cap \partial{\mathbb{D}}$ the conditions (A)-(D) also hold with $\mathcal{I}_{s}\Subset \mathcal{I}_{s-1}$. Furthermore, for any $(\phi,z)\in \Pi_{s}$,
\begin{align}
|x_{s}(\phi,z)&-x_{s-1}(\phi,z)|<\exp(-\gamma N_{s-1}/50),\label{bulk-theorem-(1)}\\
\|u^{[-N_{s}',N_{s}'']}(x_{s}(\phi,z),\cdot)&-u^{[-N_{s}',N_{s}'']}(x_{s-1}(\phi,z),\cdot)\|<\exp(-\gamma N_{s-1}/500).\label{bulk-theorem-(2)}
\end{align}
\end{theorem}

In what follows, we let $B_{0,z,\hat{h}}$ be the set from the measure estimate in condition (C), with $s=0$.

\begin{lemma}\label{Lemma4.2}
Let $\hat{h}$ as in (C), with $s=0$. Let
\begin{equation*}
\mathcal{B}_{0,z,\hat{h}}'=\{\phi\in \mathcal{I}_{0}: \mathop{\max}\limits_{|n'|,|n''|<N_{0}^{\frac{1}{2}}}\mathrm{dist}(\sigma(\mathcal{E}_{[-N_{0}+n',N_{0}+n'']}^{\beta,\eta}(x_{0}(\phi,z)+\hat{h})),z)
<\exp(-N_{0}^{\hat{\beta}})\big\}.
\end{equation*}
Then for any $z\in\mathcal{D}(z_{0},r_{0})\cap \partial{\mathbb{D}}$, the set $\mathcal{B}_{0,z,\hat{h}}'$ is contained in a semialgebraic set of degree less than $N_{0}^{20}$ and with measure less than $\exp(-N_{0}^{2\hat{\delta}})$.
\end{lemma}
\begin{proof}
Fix $z\in\mathcal{D}(z_{0},r_{0})\cap \partial{\mathbb{D}}$. By truncating $x(\cdot,z)$'s Fourier series and the Taylor series of the trigonometric functions, one can obtain a polynomial $\tilde{x}(\cdot,z)$ of degree less than $C(d)N_{0}^{4}$ such that
$$\sup_{\phi\in\mathcal{I}_{0}}|x_{0}(\phi,z)-\tilde{x}_{0}(\phi,z)|\leq \exp(-N_{0}^{2}).$$
For any $[a,b]\subset \mathbb{Z}$, $\phi\in \mathcal{I}_{0}$,
\begin{equation*}
\|\mathcal{E}_{[a,b]}^{\beta,\eta}(x_{0}(\phi,z))-\mathcal{E}_{[a,b]}^{\beta,\eta}(\tilde{x}_{0}(\phi,z))\|\leq 8|x_{0}(\phi,z)-\tilde{x}_{0}(\phi,z)|\leq \exp(-N_{0}^{2}/2).
\end{equation*}
Let $\tilde{\alpha}_{i}$, $\tilde{\mathcal{E}}_{[a,b]}^{\beta,\eta}$ be as in \eqref{VCA} and \eqref{eigenvalue-estimate}. One can obtain that
$$\|\mathcal{E}_{[a,b]}^{\beta,\eta}(x_{0}(\phi,z))-\tilde{\mathcal{E}}_{[a,b]}^{\beta,\eta}(\tilde{x}_{0}(\phi,z))\|\leq \exp(-N_{0}^{2}/4)$$
for any $[a,b]\subset \mathbb{Z}$. Let
\begin{equation*}
\tilde{\mathcal{B}}_{0,z,\hat{h}}=\{\phi\in \mathcal{I}_{0}: \mathop{\min}\limits_{|n'|,|n''|<N_{0}^{\frac{1}{2}}}\|(\tilde{\mathcal{E}}_{[-N_{0}+n',N_{0}+n'']}^{\beta,\eta}(\tilde{x}_{0}(\phi,z)+\hat{h})-z)^{-1}\|_{\mathrm{HS}}
>\exp(3N_{0}^{\hat{\beta}}/4)\},
\end{equation*}
where $\|\cdot\|_{\mathrm{HS}}$ is the Hilbert-Schmidt norm. Equivalently, for $\phi\in \tilde{\mathcal{B}}_{0,z,\hat{h}}$, one can obtain that
$$\max_{|n'|,|n''|<N_{0}^{\frac{1}{2}}}\mathrm{dist}(\sigma(\tilde{\mathcal{E}}_{[-N_{0}+n',N_{0}+n'']}^{\beta,\eta}(\tilde{x}_{0}(\phi,z)+\hat{h})),z)<\exp(-3N_{0}^{\hat{\beta}}/4).$$
Thus, $\tilde{\mathcal{B}}_{0,z,\hat{h}}$ is a semialgebraic set of degree less  than $N_{0}^{20}$ and
$$\mathcal{B}_{0,z,\hat{h}}'\subset \tilde{\mathcal{B}}_{0,z,\hat{h}} \subset \mathcal{B}_{0,z,\hat{h}}.$$
This completes the proof.
\end{proof}

\begin{lemma}\label{Lemma4.3}
For any $z\in\mathcal{D}(z_{0},r_{0})\cap \partial{\mathbb{D}}$ there exists a semialgebraic set $\mathcal{B}_{0,z,N_{1}}$,
$$\mathrm{deg}(\mathcal{B}_{0,z,N_{1}})\lesssim N_{1}N_{0}^{20},\,\,\, \mathrm{mes}(\mathcal{B}_{0,z,N_{1}})<\exp(-N_{0}^{2\hat{\delta}}/2),$$
such that for any $\phi\in \mathcal{I}_{0}\backslash \mathcal{B}_{0,z,N_{1}}$ and any $3N_{0}/2<|m|\leq N_{1}$, there exist $|n'(\phi,m)|$, $|n''(\phi,m)|<N_{0}^{\frac{1}{2}}$ such that with $J_{m}=m+[-N_{0}+n'(\phi,m),N_{0}+n''(\phi,m)]$,
$$\mathrm{dist}(\sigma(\mathcal{E}_{J_{m}}^{\beta,\eta}(x_{0}(\phi,z))),z)\geq \exp(-N_{0}^{\hat{\beta}}).$$
\end{lemma}
\begin{proof}
Take arbitrary $3N_{0}/2<|m|\leq N_{1}$. Then $0<|m-n|<3N_{1}$ for any $n\in \Upsilon_{0}$. According to the Diophantine condition, we have that
$$\mathrm{dist}(m\omega,\Upsilon_{0})>\frac{p}{(3N_{1})^{q}}\geq p(CN_{0}^{\hat{A}})^{-q}>\exp(-N_{0}^{\hat{\mu}}).$$
Then the assumption in condition (C) is satisfied by choosing $\hat{h}=m\omega$. We let $\mathcal{B}_{0,z,N_{1}}:=\mathop{\cup}\limits_{m}\tilde{\mathcal{B}}_{0,z,m\omega}$, where $\tilde{\mathcal{B}}_{0,z,m\omega}$ are the semialgebraic sets from the statement of Lemma \ref{Lemma4.2}. Thus, $\mathcal{B}_{0,z,N_{1}}$ is a semialgebraic set of degree $\lesssim N_{1}N_{0}^{20}$ and
$$\mathrm{mes}(\mathcal{B}_{0,z,N_{1}})\lesssim N_{1}\exp(-N_{0}^{2\hat{\delta}})<\exp(-\frac{1}{2}N_{0}^{2\hat{\delta}}).$$
Take $\phi\in \mathcal{I}_{0}\backslash \mathcal{B}_{0,z,N_{1}}$. Then $\phi\in \mathcal{I}_{0}\backslash \mathcal{B}_{0,z,m\omega}'$ and we can derive the conclusion according to the definition of $\mathcal{B}_{0,z,m\omega}'$.
\end{proof}

\begin{lemma}\label{Lemma4.4}
There exists $\phi_{1}\in \mathbb{T}^{d-1}$, $|\phi_{1}-\phi_{0}|\ll r_{0}^{4}$, and $|N_{1}'-N_{1}|, |N_{1}''-N_{1}|\lesssim N_{0}$ such that the following statements hold:\\
$\mathrm{(i)}$ $\mathcal{I}_{1}'\subset \mathcal{I}_{0}\backslash \mathcal{B}_{0,z_{1},N_{1}}$, $\mathcal{I}_{1}'=\phi_{1}+(-r_{1}',r_{1}')^{d-1}$, $r_{1}'=\exp(-3N_{0}^{\hat{\beta}})$, with $\mathcal{B}_{0,z_{1},N_{1}}$ as in Lemma \ref{Lemma4.3}.\\
$\mathrm{(ii)}$ There exists $k_{1}$ such that for any $\phi\in \mathcal{I}_{1}'$, $y\in \mathbb{C}^{d}$, $|y|<r_{1}'$, $z\in \partial \mathbb{D}$, $|z-z_{1}|<r_{1}',$
\begin{align}
&|z_{k_{1}}^{[-N_{1}',N_{1}'']}(x_{0}(\phi,z)+y)-z^{[-N_{0}',N_{0}'']}(x_{0}(\phi,z)+y)|<\exp(-\gamma N_{0}/40),\label{Lem4.4-(1)}\\
&|z_{j}^{[-N_{1}',N_{1}'']}(x_{0}(\phi,z)+y)-z_{k_{1}}^{[-N_{1}',N_{1}'']}(x_{0}(\phi,z)+y)|>\frac{1}{8}\exp(-N_{0}^{\hat{\beta}}),
\,\,j\neq k_{1},\label{Lem4.4-(2)}\\
&|u_{k_{1}}^{[-N_{1}',N_{1}'']}(x_{0}(\phi,z)+y,s)|<\exp(-\gamma |s|/20),\,\, |s|\geq 3N_{0}/4,\label{Lem4.4-(3)}\\
&\|u_{k_{1}}^{[-N_{1}',N_{1}'']}(x_{0}(\phi,z)+y,\cdot)-u^{[-N_{0}',N_{0}'']}(x_{0}(\phi,z)+y,\cdot)\|<\exp(-\gamma N_{0}/40).\label{Lem4.4-(4)}
\end{align}
\end{lemma}
\begin{proof}
According to the definition of the set $\mathcal{B}_{0,z_{1},N_{1}}$, there exists $\phi_{1}$ satisfying $|\phi_{1}-\phi_{0}|\ll r_{0}^{4}$ such that $\mathcal{I}_{1}'\subset \mathcal{I}_{0}\backslash \mathcal{B}_{0,z_{1},N_{1}}$. Take the intervals $J_{m}=m+[-N_{0}+n'(\phi_{1},m),N_{0}+n''(\phi_{1},m)]$ from Lemma \ref{Lemma4.3}. Define
\begin{equation}\label{Lem4.4-(5)}
[-N_{1}',N_{1}'']=[-3N_{0}/2,3N_{0}/2]\cup\mathop{\cup}\limits_{3N_{0}/2<|m|\leq N_{1}}J_{m}.
\end{equation}
Using Lemma \ref{Lemma4.3}, we have that
$$\mathrm{dist}(\sigma(\mathcal{E}_{J_{m}}^{\beta,\eta}(x_{0}(\phi_{1},z_{1}))),z_{1})\geq \exp(-N_{0}^{\hat{\beta}})\geq \exp(-N_{0}^{\tilde{\beta}}),$$
which implies that the assumptions of Lemma \ref{another-finite-scale-localization} hold.

For $\phi\in \mathcal{I}_{1}'$, $|y|<\exp(-3N_{0}^{\hat{\beta}})$, $|z-z_{1}|<\exp(-3N_{0}^{\hat{\beta}})$, we get that
$$|x_{0}(\phi,z)+y-x_{0}(\phi_{1},z_{1})|\leq \exp(CN_{0}^{\hat{\delta}})(|\phi-\phi_{1}|+|z-z_{1}|+|y|)<\exp(-2N_{0}^{\hat{\beta}}).$$
On the other hand, due to equation \eqref{(A)-(1)}, $$\big|z_{k_{0}}^{[-N_{0}',N_{0}'']}(\omega,z;x_{0}(\phi_{1},z_{1}))-z_{0}\big|<\exp(-2N_{0}^{\hat{\beta}}).$$
Since $|N_{0}'-N_{0}|<N_{0}^{\frac{1}{2}}$, $|N_{0}''-N_{0}|<N_{0}^{\frac{1}{2}}$, we have that
$$|-(N_{0}'-i)|,\,\, |N_{0}''-i|>\frac{3N_{0}}{4},\,\,i=0,1,2,3.$$
Using \eqref{(B)}, one can obtain that
$$\big|u_{k_{0}}^{[-N_{0}',N_{0}'']}(\omega,x_{0}(\phi_{1},z_{1});-(N_{0}'-i))\big|\leq\exp(-\gamma|-(N_{0}'-i)|/10)<\exp(-2N_{0}^{\hat{\beta}}),$$
and
$$\big|u_{k_{0}}^{[-N_{0}',N_{0}'']}(\omega,x_{0}(\phi_{1},z_{1});N_{0}''-i)\big|\leq\exp(-\gamma|N_{0}''-i|/10)<\exp(-2N_{0}^{\hat{\beta}}).$$
Then the assumptions of Lemma \ref{finite-scale-lemma-2} satisfied by taking $x=x_{0}(\phi,z)+y$, $x_{0}=x_{0}(\phi_{1},z_{1})$, and $z_{0}=z_{1}$. Therefore, by applying Lemma \ref{finite-scale-lemma-2}, the conclusion follows.
\end{proof}

\begin{lemma}\label{Lemma4.5}
$\mathrm{(i)}$ The function $z^{[-N_{0}',N_{0}'']}$ is analytic on $\{\hat{x}\in\mathbb{C}^{d}: |\hat{x}-x_{0}(\phi,z)|<\exp(-2N_{0}^{\hat{\delta}})\}$, for any $(\phi,z)\in \Pi_{0}$.\\
$\mathrm{(ii)}$ The function $z^{[-N_{1}',N_{1}'']}$ is analytic on $\{\hat{x}\in\mathbb{C}^{d}: |\hat{x}-x_{0}(\phi,z)|<\exp(-2N_{0}^{\hat{\beta}})\}$, for any $(\phi,z)\in \Pi_{1}''$ with $\Pi_{1}''=\mathcal{I}_{1}'\times(\mathcal{D}(z_{1},r_{1}')\cap\partial\mathbb{D}).$\\
$\mathrm{(iii)}$ For any $(\phi,z)\in \Pi_{1}''$,
\begin{equation}\label{Lem4.5-(1)}
|z^{[-N_{1}',N_{1}'']}(x_{0}(\phi,z))-z|<\exp(-\gamma N_{0}/40),
\end{equation}
\begin{equation}\label{Lem4.5-(2)}
|\partial_{z}z^{[-N_{1}',N_{1}'']}(x_{0}(\phi,z))-1|<\exp(-\gamma N_{0}/100).
\end{equation}
\end{lemma}
\begin{proof}
Applying Weierstrass' Preparation Theorem (Lemma \ref{Weierstrass}) with $z_{0}\in \partial\mathbb{D}$, $R_{0}=3\exp(-2N_{0}^{\hat{\delta}})$,
$r=\exp(-2N_{0}^{\hat{\delta}})$, $0<r_{j,0}<R_{0}$, there exist a polynomial $P(z,\underline{\omega})=z^{k}+a_{k-1}(\underline{\omega})z^{k-1}+\cdots+a_{0}(\underline{\omega})$ with $a_{j}(\underline{\omega})$ analytic in $\mathcal{P}=\prod_{j=1}^{d}\mathcal{D}(\omega_{j,0},r_{j,0})$ and an analytic function $g(z,\underline{\omega})$, $(z,\underline{\omega})\in\mathcal{D}(z_{0},r)\times \mathcal{P}$ such that the following properties hold:\\
$\mathrm{(a)}$ $\varphi_{[-N_{0}',N_{0}'']}^{\beta,\eta}=P(z,\underline{\omega})g(z,\underline{\omega})$ for any $(z,\underline{\omega})\in \mathcal{D}(z_{0},r)\times \mathcal{P}$,\\
$\mathrm{(b)}$ $g(z,\underline{\omega})\neq 0$ for any $(z,\underline{\omega})\in \mathcal{D}(z_{0},r)\times \mathcal{P}$,\\
$\mathrm{(c)}$ for any $\underline{\omega}\in \mathcal{P}$, $P(z,\underline{\omega})$ has no zeros in $\mathbb{C}\backslash\mathcal{D}(z_{0},r)$.

Due to (b), the zeros of $\varphi_{[-N_{0}',N_{0}'']}^{\beta,\eta}$ could only be the zeros of $P(z,\underline{\omega})$. In addition, it follows that the zeros of $P(z,\underline{\omega})$ lie in $\mathcal{D}(z_{0},r)$ from (c). Thus, the  zeros of $\varphi_{[-N_{0}',N_{0}'']}^{\beta,\eta}$ lie in $\mathcal{D}(z_{0},r)\cap \partial\mathbb{D}$. According to the information about the map $x_{0}$ in condition (A), the statement (i) follows. Then we can prove (ii) in the same way. Using \eqref{Lem4.4-(1)}, the estimate \eqref{Lem4.5-(1)} holds. Applying Cauchy estimates to $z^{[-N_{1}',N_{1}'']}(x_{0}(\phi,z))-z$, one can obtain that
\begin{align*}
|\partial_{z}z^{[-N_{1}',N_{1}'']}(x_{0}(\phi,z))-1|\leq \frac{\max|z^{[-N_{1}',N_{1}'']}(x_{0}(\phi,z))-z|}{\exp(-2N_{0}^{\hat{\beta}})}&\leq\frac{\exp(-\gamma N_{0}/40)}{\exp(-2N_{0}^{\hat{\beta}})}\\
&\leq \exp(-\gamma N_{0}/100).
\end{align*}

\end{proof}

In what follows, $a\dot{<}b\dot{<}c$ means that $a, b, c$ lie in $\partial \mathbb{D}$ with $b$ between $a$ and $c$ on the arc going counterclockwise from $a$ to $c$.

\begin{lemma}\label{Lemma4.6}
Let $\Pi_{1}''=\{(\phi,z)\in \mathbb{C}^{d}: |\phi-\phi_{1}|, |z-z_{1}|<\exp(-C_{0}N_{0}^{\hat{\beta}})\}$, with $C_{0}=C_{0}(d)\gg 1$. There exists a map $x_{1}:\Pi_{1}''\rightarrow \mathbb{R}^{d}$ such that
\begin{align}
z^{[-N_{1}',N_{1}'']}&(x_{1}(\phi,z))=z,\quad (\phi,z)\in \Pi_{1}'',\label{Lem4.6-(1)}\\
&x_{1}(\Pi_{1}'')\subset \mathbb{T}_{h/2}^{d}.\label{Lem4.6-(2)}
\end{align}
Furthermore, for any $(\phi,z)\in \Pi_{1}''$,
\begin{equation}\label{Lem4.6-(3)}
|x_{1}(\phi,z)-x_{0}(\phi,z)|<\exp(-\gamma N_{0}/50).
\end{equation}
\end{lemma}
\begin{proof}
Applying Lemma \ref{Lemma4.5}, one can obtain that
\begin{equation}\label{Lem4.6-(4)}
|z^{[-N_{1}',N_{1}'']}(x_{0}(\phi,z))-z|<\exp(-\gamma N_{0}/40)
\end{equation}
for any $\phi\in \mathcal{I}_{1}'$ and any $|z-z_{1}|<\exp(-3N_{0}^{\hat{\beta}})$. Given $z$ satisfying $|z-z_{1}|<\exp(-3N_{0}^{\hat{\beta}})$, we let $|z_{\pm}-z|=\exp(-\gamma N_{0}/50)$. Since $|z_{\pm}-z_{1}|<\exp(-3N_{0}^{\hat{\beta}})$, using inequality \eqref{Lem4.6-(4)} one can obtain that
\begin{align*}
|z^{[-N_{1}',N_{1}'']}(x_{0}(\phi,z_{-}))-z_{-}|<\exp(-\gamma N_{0}/40),\\
|z^{[-N_{1}',N_{1}'']}(x_{0}(\phi,z_{+}))-z_{+}|<\exp(-\gamma N_{0}/40).
\end{align*}
Due to $\exp(-\gamma N_{0}/40)<\exp(-\gamma N_{0}/50)$, we get that
$$z^{[-N_{1}',N_{1}'']}(x_{0}(\phi,z_{-}))\dot{<}z\dot{<}z^{[-N_{1}',N_{1}'']}(x_{0}(\phi,z_{+})).$$
Since $z^{[-N_{1}',N_{1}'']}(x_{0}(\phi,z_{-})), z, z^{[-N_{1}',N_{1}'']}(x_{0}(\phi,z_{+}))$ lie in $\partial\mathbb{D}$, we can regard them as the functions about the variable $\exp(i \theta)$. According to Lemma \ref{Lemma4.5}, $z^{[-N_{1}^{'},N_{1}'']}$ is analytic in the region of angle between the argument of $z^{[-N_{1}',N_{1}'']}(x_{0}(\phi,z_{-}))$ to the argument of $z^{[-N_{1}',N_{1}'']}(x_{0}(\phi,z_{+}))$. Thus, $z^{[-N_{1}',N_{1}'']}$ is continuous in the above region. By the intermediate value theorem, there exists a $\theta$ such that $z^{[-N_{1}',N_{1}'']}(e^{i\theta})=z$. Correspondingly, there exists a $\eta$ satisfying $z_{-}\dot{<}\eta\dot{<}z_{+}$ such that
\begin{equation}\label{Lem4.6-(5)}
z^{[-N_{1}',N_{1}'']}(x_{0}(\phi,\eta))=z.
\end{equation}
Let $\eta_{1}$ be the solution corresponding to $\phi=\phi_{1}$ and $z=z_{1}$. From \eqref{Lem4.5-(2)}, we get that $\partial_{\eta}z^{[-N_{1}',N_{1}'']}(x_{0}(\phi,\eta))\geq \frac{1}{2}$. According to the implicit function theorem (Lemma \ref{implicit}), for $$|\phi-\phi_{1}|, |z-z_{1}|<\exp(-C_{0}N_{0}^{\hat{\beta}}),$$
there exists a unique analytic solution $\eta(\phi,z)$ of \eqref{Lem4.6-(5)}, with $|\eta(\phi,z)-\eta_{1}|<\exp(-CN_{0}^{\hat{\beta}})$. Furthermore, $|\eta(\phi,z)-z|<2\exp(-\gamma N_{0}/40)$.

Taking $x_{1}(\phi,z)=x_{0}(\phi,\eta(\phi,z))$, \eqref{Lem4.6-(1)} and \eqref{Lem4.6-(2)} hold. Using Cauchy estimate,
\begin{align*}
|x_{1}(\phi,z)-x_{0}(\phi,z)|=|x_{0}(\phi,\eta(\phi,z))-x_{0}(\phi,z)|&<\exp(3N_{0}^{\hat{\delta}})\cdot 2\exp(-\gamma N_{0}/40)\\
&<\exp(-\gamma N_{0}/50).
\end{align*}
\end{proof}

\begin{remark}
In the above lemma, to prove the existence of map $x_{1}$, we apply the implicit function theorem for the analytic function $z^{[-N_{1}^{'},N_{1}'']}$. Since $z^{[-N_{1}^{'},N_{1}'']}$ is a complex function, it is difficult to verify the initial condition. For a MF-QP Schr\"{o}dinger operator, its eigenvalue is a real function. In \cite{GSV19-Inventiones}, the authors verified the initial condition by using the intermediate value theorem. While for complex $z^{[-N_{1}^{'},N_{1}'']}$, we can not compare the size of two function values. This brings difficulties to the analysis of the initial condition in the implicit function theorem. To overcome these difficulties, we regard $z^{[-N_{1}^{'},N_{1}'']}$ as a function about the variable $\exp(i \theta)$ and quote the notation $a\dot{<}b\dot{<}c$ in \cite[Theorem 11.1.1]{Simon-book2}.
\end{remark}

\begin{coro}\label{Coro4.7}
Using the notation of Lemma \ref{Lemma4.6}, for any $(\phi,z)\in \Pi_{1}''$,
\begin{align}
&|z-z_{j}^{[-N_{1}',N_{1}'']}(x_{1}(\phi,z))|>\frac{1}{8}\exp(-N_{0}^{\hat{\beta}})>\exp(-N_{1}^{\hat{\delta}}),
\,\,j\neq k_{1},\label{Coro4.7-(1)}\\
&|u^{[-N_{1}',N_{1}'']}(x_{1}(\phi,z),s)|<\exp(-\gamma |s|/20),\,\, |s|\geq 3N_{0}/4,\label{Coro4.7-(2)}\\
&\|u^{[-N_{1}',N_{1}'']}(x_{1}(\phi,z),\cdot)-u^{[-N_{0}',N_{0}'']}(x_{1}(\phi,z),\cdot)\|<\exp(-\gamma N_{0}/40),\label{Coro4.7-(3)}\\
&\|u^{[-N_{1}',N_{1}'']}(x_{1}(\phi,z),\cdot)-u^{[-N_{0}',N_{0}'']}(x_{0}(\phi,z),\cdot)\|<\exp(-\gamma N_{0}/500).\label{Coro4.7-(4)}
\end{align}
\end{coro}
\begin{proof}
Applying Lemma \ref{Lemma4.4} with $y=x_{1}(\phi,z)-x_{0}(\phi,z)$, \eqref{Lem4.4-(2)}, \eqref{Lem4.4-(3)} and \eqref{Lem4.4-(4)} imply that the estimates \eqref{Coro4.7-(1)}, \eqref{Coro4.7-(2)} and \eqref{Coro4.7-(3)} hold. Then we need to verify the last estimate.
According to the proof of Lemma \ref{finite-scale-lemma-2}, one can obtain that
\begin{align*}
&\|(\mathcal{E}_{[-N_{0}',N_{0}'']}^{\beta,\eta}(x_{0}(\phi,z))-z^{[-N_{1}',N_{1}'']}(x_{0}(\phi,z)))u^{[-N_{1}',N_{1}'']}(x_{0}(\phi,z))\|\\
&\leq \|(\mathcal{E}_{[-N_{0}',N_{0}'']}^{\beta,\eta}(x_{0}(\phi,z))-\mathcal{E}_{[-N_{0}',N_{0}'']}^{\beta,\eta}(x_{1}(\phi,z)))u^{[-N_{1}',N_{1}'']}(x_{0}(\phi,z))\|\\
&\quad+\|(\mathcal{E}_{[-N_{0}',N_{0}'']}^{\beta,\eta}(x_{1}(\phi,z))-z^{[-N_{1}',N_{1}'']}(x_{0}(\phi,z)))u^{[-N_{1}',N_{1}'']}(x_{0}(\phi,z))\|\\
&\lesssim 8|x_{1}(\phi,z)-x_{0}(\phi,z)|+\exp(-\gamma(N_{0}-N_{0}^{1/2})/20)\\
&<\exp(-\gamma N_{0}/60)=:\tilde{\varepsilon}.
\end{align*}
Furthermore, with the aid of Taylor's formula and \eqref{Lem4.6-(3)}, we have that
\begin{align*}
&|z^{[-N_{1}',N_{1}'']}(x_{1}(\phi,z))-z^{[-N_{0}',N_{0}'']}(x_{0}(\phi,z))|\\
&|z^{[-N_{1}',N_{1}'']}(x_{1}(\phi,z))-z^{[-N_{0}',N_{0}'']}(x_{1}(\phi,z))|+
|z^{[-N_{0}',N_{0}'']}(x_{1}(\phi,z))-z^{[-N_{0}',N_{0}'']}(x_{0}(\phi,z))|\\
&<\exp(-\gamma N_{0}/40)+|\langle z^{[-N_{0}',N_{0}'']}(x_{0}(\phi,z),x_{1}(\phi,z)-x_{0}(\phi,z))\rangle|\\
&\quad +C \exp(3N_{0}^{\hat{\delta}})\|x_{1}(\phi,z)-x_{0}(\phi,z)\|^{2}\\
&<\exp(-\gamma N_{0}/70)=:\hat{\varepsilon}.
\end{align*}
According to Lemma \ref{eigenvector}, we have that
\begin{equation*}
\|u^{[-N_{1}',N_{1}'']}(x_{1}(\phi,z),\cdot)-u^{[-N_{0}',N_{0}'']}(x_{0}(\phi,z),\cdot)\|<\sqrt{2}\hat{\varepsilon}^{-1}\tilde{\varepsilon}<\exp(-\gamma N_{0}/500).
\end{equation*}
\end{proof}

Then we verify condition (C) with $s=1$. Let $\mathcal{I}_{0}'=\{\phi\in \mathbb{C}^{d-1}: |\phi-\phi_{0}|<r_{0}^{4}\}$.
\begin{lemma}\label{Lemma4.7}
Let $\hat{h}\in \mathbb{C}^{d}$, $\exp(-N_{1}^{\hat{\mu}})\leq \|\hat{h}\|<\exp(-N_{0}^{\hat{\mu}})$, and $z\in \mathcal{D}(z_{1},r_{1})\cap \partial\mathbb{D}$. Then for any $\hat{\nu}>0$,
\begin{equation*}
\mathrm{mes}\{\phi\in\mathcal{I}_{0}'/6:\log |z^{[-N_{0}',N_{0}'']}(x_{0}(\phi,z)+\hat{h})-z|\leq -N_{1}^{\hat{\mu}+\hat{\nu}} \}<C(d)\exp(-N_{1}^{\hat{\nu}/(d-1)}).
\end{equation*}
\end{lemma}
\begin{proof}
According to Cauchy estimates,
$$\Big|\frac{d^{2}}{d\hat{h}^{2}}z^{[-N_{0}',N_{0}'']}(x_{0}(\phi,z)+\hat{h})\Big|\leq \exp(3N_{0}^{\hat{\delta}}).$$
By Taylor's formula,
\begin{equation}\label{Lem4.7-(1)}
z^{[-N_{0}',N_{0}'']}(x_{0}(\phi,z)+\hat{h})-z^{[-N_{0}',N_{0}'']}(x_{0}(\phi,z))=\langle \nabla z^{[-N_{0}',N_{0}'']}(x_{0}(\phi,z)),\hat{h} \rangle+O(\exp(3N_{0}^{\hat{\delta}})\|\hat{h}\|^{2}).
\end{equation}
Take $h_{0}:=\frac{\hat{h}}{\|\hat{h}\|}$. Due to condition (D), there exists $\hat{\phi}_{0}$ satisfying $|\hat{\phi}_{0}-\hat{\phi}|\ll r_{0}^{4}$ such that
$$|\langle \nabla z^{[-N_{0}',N_{0}'']}(x_{0}(\hat{\phi}_{0},z_{1})),h_{0} \rangle|\geq \exp(-N_{0}^{\hat{\mu}}/2).$$
Using Cauchy estimates, one can obtain that
\begin{align*}
|\nabla z^{[-N_{0}',N_{0}'']}(x_{0}(\hat{\phi}_{0},z_{1}))-\nabla z^{[-N_{0}',N_{0}'']}(x_{0}(\hat{\phi}_{0},z))|
&\leq \exp(3N_{0}^{\hat{\delta}})|x_{0}(\hat{\phi}_{0},z_{1})-x_{0}(\hat{\phi}_{0},z)|\\
&\leq \exp(CN_{0}^{\hat{\delta}})|z-z_{1}|\\
&\leq \exp(CN_{0}^{\hat{\delta}})\exp(-N_{1}^{\hat{\delta}}).
\end{align*}
Thus, for any $z\in \mathcal{D}(z_{1},r_{1})\cap \partial \mathbb{D}$,
$$|\langle \nabla z^{[-N_{0}',N_{0}'']}(x_{0}(\hat{\phi}_{0},z)),h_{0} \rangle|\geq \exp(-N_{0}^{\hat{\mu}}/2).$$
Since $h_{0}:=\frac{\hat{h}}{\|\hat{h}\|}$,
$$|z^{[-N_{0}',N_{0}'']}(x_{0}(\hat{\phi}_{0},z)+\hat{h})-z|\geq \|\hat{h}\|\exp(-N_{0}^{\hat{\mu}}/2)\geq \exp(-2N_{1}^{\hat{\mu}}).$$
In other words,
$$\log|z^{[-N_{0}',N_{0}'']}(x_{0}(\hat{\phi}_{0},z)+\hat{h})-z|\geq -2N_{1}^{\hat{\mu}}=:m.$$
On the other hand,
$$\sup\log|z^{[-N_{0}',N_{0}'']}(x_{0}(\hat{\phi}_{0},z)+\hat{h})-z|\leq 0=:M.$$
Applying Cartan's estimate with $C_{d}=\frac{1}{2}$ and $H=N_{1}^{\hat{\nu}}$, there exists a set $\hat{\Omega}\subset \mathcal{I}_{0}'$, $\hat{\Omega}\in \mathrm{Car}_{d-1}(H^{\frac{1}{d-1}},K)$ with $K=C_{d}H(M-m)$, such that
\begin{equation*}
\log |z^{[-N_{0}',N_{0}'']}(x_{0}(\hat{\phi}_{0},z)+\hat{h})-z|>M-C_{d}H(M-m)>-N_{1}^{\hat{\mu}+\hat{\nu}}
\end{equation*}
for any $\phi\in \frac{1}{6}\mathcal{I}_{0}'\backslash \hat{\Omega}$.

Therefore, according to Lemma \ref{Cartan-measure}, we have that
\begin{equation*}
\mathrm{mes}\{\phi\in\mathcal{I}_{0}'/6:\log |z^{[-N_{0}',N_{0}'']}(x_{0}(\phi,z)+\hat{h})-z|\leq -N_{1}^{\hat{\mu}+\hat{\nu}} \}<C(d)\exp(-N_{1}^{\hat{\nu}/(d-1)}).
\end{equation*}
\end{proof}

\begin{lemma}\label{Lemma4.8}
Let $\hat{h}\in \mathbb{T}^{d}$ such that $\mathrm{dist}(\hat{h},\Upsilon_{1})\geq \exp(-N_{1}^{\hat{\mu}})$ and
$$\mathcal{B}_{1,z,\hat{h}}''=\{\phi\in \mathcal{I}_{1}'':\max_{|n'|,|n''|<N_{1}^{\frac{1}{2}}}\mathrm{dist}(\sigma(\mathcal{E}_{-N_{1}+n',N_{1}+n''}^{\beta,\eta}x_{1}(\phi,z)+\hat{h}),z)
<\exp(-N_{1}^{\hat{\beta}}/2)\},$$
where $\mathcal{I}''=\{\phi\in \mathbb{C}^{d-1}:|\phi-\phi_{1}|<\exp(-C_{0}N_{0}^{\hat{\beta}})\}$. Then for any $z\in \mathcal{D}(z_{1},r_{1})\cap \partial \mathbb{D}$, $\mathrm{mes}(\mathcal{B}_{1,z,\hat{h}}'')<\exp(-N_{1}^{2\hat{\delta}})$.
\end{lemma}
\begin{proof}
Take $|m_{1}|\leq3N_{1}/2$, $h_{1}\in\mathbb{R}^{d}$ such that
$$\mathrm{dist}(\hat{h},\Upsilon_{1})=\|h_{1}\|,\quad h_{1}=\hat{h}-m_{1}\omega\, (\mathrm{mod}\,\mathbb{Z}^{d}).$$
Since $-m+[-3N_{0}/2,3N_{0}/2]\subset[-3N_{1}/2,3N_{1}/2]$ for any $m\in [-N_{1},N_{1}]$, one can obtain that
\begin{equation}\label{Lem4.8-(1)}
\mathrm{dist}(\hat{h}+m\omega,\Upsilon_{0})=\mathrm{dist}(\hat{h},-m\omega+\Upsilon_{0})\geq \mathrm{dist}(\hat{h},\Upsilon_{1})=\|h_{1}\|.
\end{equation}
Then we consider the following two cases:\\
{\bf{Case 1.}} $|m+m_{1}|>3N_{0}/2$\\
In this case, according to the standard Diophantine condition, we get that
\begin{equation}\label{Lem4.8-(2)}
\|\hat{h}+m\omega-n\omega\|=\|h_{1}+(m+m_{1}-n)\omega\|\geq\|(m+m_{1}-n)\omega\|-\|h_{1}\|\geq p(CN_{1})^{-q}-\|h_{1}\|
\end{equation}
for any $n$ satisfying $n\omega\in \Upsilon_{0}$.

If $\|h_{1}\|\geq \exp(-N_{0}^{\hat{\mu}})$, then $\mathrm{dist}(\hat{h}+m\omega,\Upsilon_{0})\geq \exp(-N_{0}^{\hat{\mu}})$ follows from \eqref{Lem4.8-(1)}.

If $\|h_{1}\|< \exp(-N_{0}^{\hat{\mu}})$, then $\mathrm{dist}(\hat{h}+m\omega,\Upsilon_{0})\geq \exp(-N_{0}^{\hat{\mu}})$ follows from \eqref{Lem4.8-(2)}.

Write $\hat{h}+m\omega=(m+\frac{\hat{h}}{\omega})\omega=:\tilde{h}\omega$. According to condition (C), for each $\phi\in \mathcal{I}_{0}\backslash \mathcal{B}_{0,z_{1},\lceil\tilde{h}\rceil}$ with $\mathcal{B}_{0,z_{1},\lceil\tilde{h}\rceil}$ as in Lemma \ref{Lemma4.3}, there exist $|n'|, |n''|<N_{0}^{\frac{1}{2}}$ such that
$$\mathrm{dist}(\sigma(\mathcal{E}_{J_{m}(\phi)}^{\beta,\eta}(x_{0}(\phi,z_{1})+\hat{h})),z_{1})\geq\exp(-N_{0}^{\hat{\beta}}/2),$$
with $J_{m}(\phi)=m+[-N_{0}+n',N_{0}+n'']$.

Thus, for any $z\in\mathcal{D}(z_{1},r_{1})\cap\partial\mathbb{D}$, we have that
$$\mathrm{dist}(\sigma(\mathcal{E}_{J_{m}(\phi)}^{\beta,\eta}(x_{0}(\phi,z)+\hat{h})),z)\geq\exp(-N_{0}^{\nu/2}).$$
Since it follows from condition (C) that $\mathrm{mes}(\mathcal{B}_{0,z_{1},\lceil\tilde{h}\rceil})<\exp(-N_{0}^{2\hat{\delta}})$, there exists $\phi_{0,m}\in \mathcal{I}_{0}\backslash\mathcal{B}_{0,z_{1},\lceil\tilde{h}\rceil}$, $|\phi_{0,m}-\phi_{0}|\ll r_{0}^{4}$. Define $J_{m}:=J_{m}(\phi_{0,m})$. Applying the spectral form of (LDT), we have that
$$\log|\varphi_{J_{m}}^{\beta,\eta}(x_{0}(\phi_{0,m},z)+\hat{h},z)|>|J_{m}|L_{|J_{m}|}(z)-|J_{m}|^{1-\tau/2}=:m.$$
On the other hand, from Remark \ref{Remark3.7}, we have that
$$\sup_{\phi\in \mathcal{I}_{0}'}\log|\varphi_{J_{m}}^{\beta,\eta}(x_{0}(\phi,z)+\hat{h},z)|\leq |J_{m}|L_{|J_{m}|}(z)+C|J_{m}|^{1-\tau}=:M.$$
Using Cartan's estimate to $\varphi_{J_{m}}^{\beta,\eta}(x_{0}(\phi,z)+\hat{h},z)$ with $H=N_{0}^{\tau}$, $C_{d-1}H=1$, there exists a set $\mathcal{B}_{0,z,m}'\subset\mathcal{I}_{0}'$, $\mathcal{B}_{0,z,m}'\in \mathrm{Car}_{d-1}(H^{\frac{1}{d-1}},K)$ ($K=C_{d-1}H(M-m)$) such that
\begin{align*}
\log|\varphi_{J_{m}}^{\beta,\eta}(x_{0}(\phi,z)+\hat{h},z)|\geq M-C_{d-1}H(M-m)&\geq|J_{m}|L_{|J_{m}|}(z)-|J_{m}|^{1-\tau/4}\\
&>|J_{m}|L(z)-|J_{m}|^{1-\tau/4}
\end{align*}
for any $\phi\in \frac{1}{6}\mathcal{I}_{0}'\backslash \mathcal{B}_{0,z,m}'$. Thus, one can obtain that
\begin{equation}\label{Lem4.8-(3)}
\mathrm{mes}\{\phi\in \frac{1}{6}\mathcal{I}_{0}':\log|\varphi_{J_{m}}^{\beta,\eta}(x_{0}(\phi,z)+\hat{h},z)|<|J_{m}|L(z)-|J_{m}|^{1-\tau/4} \}<C(d)\exp(-N_{0}^{\tau/(d-1)})
\end{equation}
and
$$\mathcal{B}_{0,z,m}'=\{\phi\in \frac{1}{6}\mathcal{I}_{0}':\log|\varphi_{J_{m}}^{\beta,\eta}(x_{0}(\phi,z)+\hat{h},z)|<|J_{m}|L(z)-|J_{m}|^{1-\tau/4} \}.$$
Let $\mathcal{B}_{0,z,N_{1}}'=\mathop{\cup}\limits_{-N_{1}\leq m \leq N_{1},|m+m_{1}|>3N_{0}/2}\mathcal{B}_{0,z,m}'$. We have that
$$\mathrm{mes}(\mathcal{B}_{0,z,N_{1}}')\leq 2N_{1}C(d)\exp(-N_{0}^{\tau/(d-1)})<\exp(-N_{0}^{\tau/(d-1)}/2)\ll\exp(-N_{1}^{2\hat{\delta}}).$$
{\bf{Case 2.}} $|m+m_{1}|\leq3N_{0}/2$\\
In this case, we only need to consider $m=-m_{1}$. Then for any $m_{1}\in[-N_{1},N_{1}]$, \eqref{Lem4.8-(1)} holds. Similarly, we consider the following two cases.

If $\|h_{1}\|\geq \exp(-N_{0}^{\hat{\mu}})$, then $\mathrm{dist}(\hat{h}+m\omega,\Upsilon_{0})\geq \exp(-N_{0}^{\hat{\mu}})$. Besides, there exists an interval $J_{-m}$, such that \eqref{Lem4.8-(3)} holds with $m=-m_{1}$. In this case, we let $\mathcal{B}_{0,z,-m_{1}}'$ be the set from \eqref{Lem4.8-(3)}.

Suppose $\|h_{1}\|< \exp(-N_{0}^{\hat{\mu}})$. Let $J_{-m_{1}}:=-m_{1}+[-N_{0}',N_{0}'']$. Then one can obtain that
\begin{equation*}
\sigma(\mathcal{E}_{J_{-m_{1}}}^{\beta,\eta}(x+\hat{h}))=\sigma(\mathcal{E}_{[-N_{0}',N_{0}'']}^{\beta,\eta}(x
+\hat{h}-m_{1}\omega))=\sigma(\mathcal{E}_{[-N_{0}',N_{0}'']}^{\beta,\eta}(x+h_{1})).
\end{equation*}
In this case, we let
\begin{equation}\label{Lem4.8-(4)}
\mathcal{B}_{0,z,-m_{1}}'=\{\phi\in \frac{1}{6}\mathcal{I}_{0}':|z^{[-N_{0}',N_{0}]''}(x_{0}(\phi,z)+h_{1})-z|\leq\exp(-N_{1}^{\hat{\mu}+\hat{\nu}})\}
\end{equation}
with $\hat{\nu}=3(d-1)\hat{\delta}$. According to Lemma \ref{Lemma4.7},
$$\mathrm{mes}(\mathcal{B}_{0,z,-m_{1}}')<C(d)\exp(-N_{1}^{\hat{\nu}/(d-1)})<\exp(-N_{1}^{2\hat{\delta}}).$$
Then from \eqref{Lem4.8-(4)}, one can obtain that
$$\mathrm{dist}(\sigma(\mathcal{E}_{J_{-m_{1}}}^{\beta,\eta}(x_{0}(\phi,z)+\hat{h})),z)>\exp(-N_{1}^{\hat{\mu}+\hat{\nu}})>\exp(-|J_{-m_{1}}|^{\nu/2})$$
for any $\phi\in \frac{1}{6}\mathcal{I}_{0}'\backslash \mathcal{B}_{0,z,-m_{1}}'$.

According to the spectral form of (LDT), we get that
$$\log|\varphi_{J_{m_{1}}}^{\beta,\eta}(x_{0}(\phi,z)+\hat{h})|>|J_{-m_{1}}|L(z)-|J_{-m_{1}}|^{1-\tau/2} $$
for any $\phi\in \frac{1}{6}\mathcal{I}_{0}'\backslash \mathcal{B}_{0,z,-m_{1}}'$.

Thus, in either case we can find an interval $J_{-m_{1}}$ and get a similar conclusion.

Let
\begin{align*}
I:=
\begin{cases}
J_{-m_{1}}\cup(\mathop{\cup}\limits_{-N_{1}\leq m\leq N_{1},|m+m_{1}|>3N_{0}/2}J_{m}), \quad &m_{1}\in[-N_{1},N_{1}], \\
\mathop{\cup}\limits_{-N_{1}+2N_{0}\leq m\leq N_{1}-2N_{0}}J_{m},& m_{1}\notin[-N_{1},N_{1}].
\end{cases}
\end{align*}
According to the covering form of (LDT), we have that
$$\mathrm{dist}(\sigma(\mathcal{E}_{I}^{\beta,\eta}(x_{0}(\phi,z)+\hat{h})),z)\geq \exp(-2\max_{m}|J_{m}|^{1-\tau/4})>\exp(-4N_{0}^{1-\tau/4})$$
for any $\phi\in  \frac{1}{6}\mathcal{I}_{0}'\backslash (\mathcal{B}_{0,z,N_{1}}'\cup\mathcal{B}_{0,z,-m_{1}}')$. Due to \eqref{Lem4.6-(3)},
$$\mathrm{dist}(\sigma(\mathcal{E}_{I}^{\beta,\eta}(x_{1}(\phi,z)+\hat{h})),z)\gtrsim \exp(-4N_{0}^{1-\tau/4})\gg \exp(-N_{1}^{\hat{\beta}}/2)$$
for any $\phi\in  \mathcal{I}_{1}''\backslash (\mathcal{B}_{0,z,N_{1}}'\cup\mathcal{B}_{0,z,-m_{1}}')$.

Therefore, $\mathcal{B}_{1,z,\hat{h}}''\subset(\mathcal{B}_{0,z,N_{1}}'\cup\mathcal{B}_{0,z,-m_{1}}')$ and the statement follows.

\end{proof}

\begin{lemma}\label{Lemma4.10}
For any $(\phi,z)\in \Pi_{1}''$,
\begin{equation*}
|\nabla z^{[-N_{1}',N_{1}'']}(x_{1}(\phi,z))-\nabla z^{[-N_{0}',N_{0}'']}(x_{0}(\phi,z))|<\exp(-c_{0}\gamma N_{0})
\end{equation*}
for some suitable $c_{0}=c_{0}(d)$.
\end{lemma}
\begin{proof}
Applying Cauchy estimates and \eqref{Lem4.6-(3)}, one can obtain that
\begin{align*}
|\nabla z^{[-N_{0}',N_{0}'']}(x_{1}(\phi,z))-\nabla z^{[-N_{0}',N_{0}'']}(x_{0}(\phi,z))|&\leq \exp(3N_{0}^{\hat{\delta}})|x_{1}(\phi,z)-x_{0}(\phi,z)|\\
&\leq\exp(-\gamma N_{0}/60).
\end{align*}
Combining \eqref{Lem4.4-(1)} and Cauchy estimates, we have that
\begin{equation*}
|\nabla z^{[-N_{1}',N_{1}'']}(x_{1}(\phi,z))-\nabla z^{[-N_{0}',N_{0}'']}(x_{1}(\phi,z))|\leq \exp(-c(d)\gamma N_{0})
\end{equation*}
for some appropriate $c(d)$.

Therefore, for some suitable $c_{0}=c_{0}(d)$, we have that
\begin{align*}
&|\nabla z^{[-N_{1}',N_{1}'']}(x_{1}(\phi,z))-\nabla z^{[-N_{0}',N_{0}'']}(x_{0}(\phi,z))|\\
&\leq |\nabla z^{[-N_{1}',N_{1}'']}(x_{1}(\phi,z))-\nabla z^{[-N_{0}',N_{0}'']}(x_{1}(\phi,z))|+|\nabla z^{[-N_{0}',N_{0}'']}(x_{1}(\phi,z))\\
&\quad -\nabla z^{[-N_{0}',N_{0}'']}(x_{0}(\phi,z))|\\
&<\exp(-c_{d}\gamma N_{0})+\exp(-\gamma N_{0}/60)<\exp(-c_{0}\gamma N_{0}).
\end{align*}

\end{proof}

\begin{lemma}\label{Lemma4.11}
Let $h_{0}\in \mathbb{C}^{d}$ be a unit vector. Then for any $z\in \mathcal{D}(z_{1},r_{1})\cap\partial\mathbb{D}$,
\begin{equation*}
\mathrm{mes}\{\phi\in \mathcal{I}_{1}'':\log|\langle\nabla z^{[-N_{1}',N_{1}'']}(x_{1}(\phi,z)),h_{0}\rangle|<-N_{1}^{\hat{\mu}}/2\}<\exp(-N_{1}^{2\hat{\delta}}).
\end{equation*}
\end{lemma}
\begin{proof}
Due to condition (D), there exists $\hat{\phi}_{0}$ satisfying $|\hat{\phi}_{0}-\phi_{0}|\ll r_{0}^{4}$ such that
\begin{equation*}
\log|\langle\nabla z^{[-N_{0}',N_{0}'']}(x_{0}(\hat{\phi}_{0},z)),h_{0}\rangle|\geq-N_{0}^{\hat{\mu}}/2=:m.
\end{equation*}
On the other hand,
\begin{equation*}
\sup\log|\langle\nabla z^{[-N_{0}',N_{0}'']}(x_{0}(\phi,z)),h_{0}\rangle|\leq 0=:M.
\end{equation*}
Applying Cartan's estimate with $C_{d}=2$ and $H=N_{0}^{\nu_{0}}$ ($\nu_{0}=3(d-1)\hat{\beta}^{-1}\hat{\delta}$), there exists a set $\tilde{\Omega}\subset\mathcal{I}_{0}'$, $\tilde{\Omega}\in \mathrm{Car}_{d-1}(H^{\frac{1}{d-1}},K)$ with $K=C_{d}H(M-m)$, such that
\begin{equation*}
\log|\langle\nabla z^{[-N_{0}',N_{0}'']}(x_{0}(\phi,z)),h_{0}\rangle|>M-C_{d}H(M-m)=-N_{0}^{\hat{\mu}+\nu_{0}}
\end{equation*}
for any $\phi\in \frac{1}{6}\mathcal{I}_{0}'\backslash \tilde{\Omega}$.

Thus, using Lemma \ref{Cartan-measure}, we have that
\begin{align*}
\mathrm{mes}\{\phi\in\frac{1}{6}\mathcal{I}_{0}': \log|\langle\nabla z^{[-N_{0}',N_{0}'']}(x_{0}(\phi,z)),h_{0}\rangle|\leq-N_{0}^{\hat{\mu}+\nu_{0}}\}&<C(d)\exp(-N_{0}^{\nu_{0}/(d-1)})\\
&<\exp(-N_{1}^{2\hat{\delta}}).
\end{align*}
In addition, $\tilde{\Omega}=\{\phi\in\frac{1}{6}\mathcal{I}_{0}': \log|\langle\nabla z^{[-N_{0}',N_{0}'']}(x_{0}(\phi,z)),h_{0}\rangle|\leq-N_{0}^{\hat{\mu}+\nu_{0}}\}$.

Since $\mathcal{I}_{1}''=\{\phi\in\mathbb{R}^{d-1}:|\phi-\phi_{1}|<\exp(-C_{0}N_{0}^{\hat{\beta}})\}$, $\hat{\delta}\ll \hat{\beta}$, then $\mathcal{I}_{1}''\subset\frac{1}{6}\mathcal{I}_{0}'$. It follows from Lemma \ref{Lemma4.10} that
\begin{equation*}
\log|\langle\nabla z^{[-N_{0}',N_{0}'']}(x_{0}(\phi,z)),h_{0}\rangle|\geq-2N_{0}^{\hat{\mu}+\nu_{0}}\geq -N_{1}^{\hat{\mu}}/2
\end{equation*}
holds for any $\phi\in \mathcal{I}_{1}''\backslash\tilde{\Omega}$. Therefore, the conclusion follows.
\end{proof}

In the end of this section, we prove Theorem \ref{bulk-theorem}.
\begin{proof}
The existence of $\phi_{1}$ was obtained in \ref{Lemma4.4}. Combining the definition of the sets $\mathcal{I}_{0}'$, $\mathcal{I}_{1}''$ and Lemma \ref{Lemma4.11}, we have  that $\mathcal{I}_{1}\Subset \mathcal{I}_{1}''\subset \frac{1}{6}\mathcal{I}_{0}'\subset \mathcal{I}_{0}$.  Conditions (A), (B) and the estimates \eqref{bulk-theorem-(1)},  \eqref{bulk-theorem-(2)} follow from Lemma \ref{Lemma4.6} and Corollary \ref{Coro4.7}. Condition (C) follows from Lemma \ref{Lemma4.8}. Condition (D) follows from Lemma \ref{Lemma4.11}.
\end{proof}

\section{Proof of the main result}
The following proposition is a more detailed version of Theorem \ref{main-theorem}.
\begin{prop}\label{Prop5.1}
Assume the notation of the inductive conditions (A)-(D) from Section 4. Let $z_{0}\in \partial\mathbb{D}$, $N_{0}\geq 1$. Assume that $L(z)>\gamma>0$ for $z\in\mathcal{D}(z_{0},2r_{0})\cap\partial \mathbb{D}$ with $r_{0}=\exp(-N_{0}^{\hat{\delta}})$. If $N_{0}\geq C(p,q,\hat{\beta})$ and conditions (A)-(D) hold with $s=0$ for the given $z_{0}$, then $\bar{\mathcal{D}}(z_{0},r_{0})\cap\partial \mathbb{D}\subset\sigma(\mathcal{E}(x))$, where $\bar{\mathcal{D}}(z_{0},r_{0})=\{z\in \mathbb{C}: |z-z_{0}|\leq r_{0}\}$.
\end{prop}
\begin{proof}
Let $z\in \mathcal{D}(z_{0},r_{0})\cap\partial \mathbb{D}$ be arbitrary. By iteratively applying Theorem \ref{bulk-theorem} with $z_{s}=z$ for  $s\geq 1$, and noting that $\mathcal{I}_{s}\Subset\mathcal{I}_{s-1}$, we obtain a sequence of nested intervals. Consequently, there exists $\hat{\phi}\in \cap_{s}\mathcal{I}_{s}$. By \eqref{bulk-theorem-(1)}, there exists $x(z)$ such that
\begin{equation*}
|x(z)-x_{s}(\hat{\phi},z)|<2\exp(-\gamma N_{s}/50),\quad s\geq 0.
\end{equation*}
In addition, using \eqref{bulk-theorem-(2)}, there exists a normalized vector $u(z,\cdot)$ such that
\begin{equation*}
\|u(z,\cdot)-u^{[-N_{s}',N_{s}'']}((\hat{\phi},z),\cdot)\|<2\exp(-\gamma N_{s}/500),\quad s\geq 0.
\end{equation*}
Furthermore,
\begin{align*}
&\|(\mathcal{E}(x_{s}(\hat{\phi},z))-z)u^{[-N_{s}',N_{s}'']}(x_{s}(\hat{\phi},z),\cdot)\|\\
&\leq\|(\mathcal{E}(x_{s}(\hat{\phi},z))-\mathcal{E}(x(z)))u^{[-N_{s}',N_{s}'']}(x_{s}(\hat{\phi},z),\cdot)\| +\|(\mathcal{E}(x(z))-z)u^{[-N_{s}',N_{s}'']}(x_{s}(\hat{\phi},z),\cdot)\|\\
&\leq 8|x(z)-x_{s}(\hat{\phi},z)|+\|(\mathcal{E}(x(z))-z)u^{[-N_{s}',N_{s}'']}(x_{s}(\hat{\phi},z),\cdot)\|\\
&\lesssim \exp(-\gamma N_{s}/60).
\end{align*}
Then one can obtain that
\begin{equation*}
\|(\mathcal{E}(x(z))-z)u(z,\cdot)\|\lesssim \exp(-\gamma N_{s}/70),\quad s\geq 0.
\end{equation*}
It follows that $\mathcal{E}(x(z))u(z,\cdot)=zu(z,\cdot)$, i.e., $z\in \sigma(\mathcal{E}(x(z)))$. Then the statement of the theorem follows from the arbitrariness of $z$ and the closeness of $\sigma(\mathcal{E}(x))$.
\end{proof}

\begin{appendices}

\section{Appendix}
In this section, some useful definitions and lemmas are given.

\subsection{Some useful notations}
\begin{definition 3}\cite[Definition 2.3]{BGS02-Acta}
For any positive numbers $a, b$ the notation $a\lesssim b$ means $Ca\leq b$ for some constant $C>0$. By $a\ll b$ we mean that the constant $C$ is very large. If both $a\lesssim b$ and $a\gtrsim b$, then we write $a\asymp b$.
\end{definition 3}
\subsection{Avalanche principle}
\begin{lemma}\label{avalanche-principle}\cite[Proposition 2.2]{GS01-Annals}
Let $A_{1},\ldots,A_{m}$ be a sequence of arbitrary unimodular $2\times 2$-matrices. Suppose that
\begin{equation}\label{AP-1}
\min_{1\leq j\leq m}\|A_{j}\|\geq \mu\geq m
\end{equation}
and
\begin{equation}\label{AP-2}
\max_{1\leq j< m}\big[\log\|A_{j+1}\|+\log\|A_{j}\|-\log\|A_{j+1}A_{j}\|\big]<\frac{1}{2} \log\mu.
\end{equation}
Then
\begin{equation}\label{AP-3}
\Big|\log\|A_{m}\cdots A_{1}\|+\sum_{j=2}^{m-1}\log\|A_{j}\|-\sum_{j=1}^{m-1}\log\|A_{j+1}A_{j}\|\Big|<C_{A}\frac{m}{\mu},
\end{equation}
where $C_{A}$ is an absolute constant.
\end{lemma}
\subsection{Properties of special matrices}
\begin{lemma}\label{eigenvector}\cite[Corollary 9.2]{Zhang-Piao}
Let $A$ be an $N\times N$ unitary matrix. Let $z\in \partial \mathbb{D}$, $\tilde{\varepsilon}>0\in \mathbb{R}$, and suppose there exists $\phi\in \mathbb{C}^{N}$, $\|\phi\|=1$, such that
$$\|(A-z)\phi\|<\tilde{\varepsilon}.$$
Then the following statements hold.\\
$\mathrm{(a)}$ There exists a normalized eigenvector $\psi$ of $A$ with an eigenvalue $z_{0}$ such that
$$z_{0}\in \mathcal{D}(z,\tilde{\varepsilon}\sqrt{2})\cap \partial \mathbb{D},$$
$$|\langle\phi,\psi\rangle|\geq(2N)^{-1/2}.$$
$\mathrm{(b)}$ If in addition there exists $\hat{\varepsilon}>\tilde{\varepsilon}$ such that the subspace of the eigenvectors of $A$ with eigenvalues falling into the interval $\mathcal{D}(z,\hat{\varepsilon})\cap \partial \mathbb{D}$ is at most of dimension one, then there exists a normalized eigenvector $\psi$ of $A$ with an eigenvalue $z_{0}\in \mathcal{D}(z,\tilde{\varepsilon})\cap \partial \mathbb{D}$, such that
$$\|\phi-\psi\|<\sqrt{2}\hat{\varepsilon}^{-1}\tilde{\varepsilon}.$$
\end{lemma}

\subsection{Weierstrass' preparation theorem}
Consider an analytic function $f(z,\omega_{1},\ldots,\omega_{d})$ defined in a polydisk
\begin{equation*}
\mathcal{P}_{*}=\mathcal{D}(z_{0},R_{0})\times \prod_{j=1}^{d}\mathcal{D}(\omega_{j,0},R_{0}),\,\,z_{0},\,\omega_{j,0}\in \mathbb{C}, \,\, R_{0}>0.
\end{equation*}
\begin{lemma}\cite[Lemma 2.28]{GSV16-arXiv}\label{Weierstrass}
Assume that $f(\cdot,\omega_{1},\ldots,\omega_{d})$ has no zeros on some circle
\begin{equation*}
\{z:|z-z_{0}|=r\},\,\, 0<r<R_{0}/2,
\end{equation*}
for any $\underline{\omega}=(\omega_{1},\ldots,\omega_{d})\in \mathcal{P}=\mathop{\prod}\limits_{j=1}^{d}\mathcal{D}(\omega_{j,0},r_{j,0})$ where $0<r_{j,0}<R_{0}$. Then there exist a polynomial $P(z,\underline{\omega})=z^{k}+a_{k-1}(\underline{\omega})z^{k-1}+\cdots+a_{0}(\underline{\omega})$ with $a_{j}(\underline{\omega})$ analytic in $\mathcal{P}$ and an analytic function $g(z,\underline{\omega})$, $(z,\underline{\omega})\in\mathcal{D}(z_{0},r)\times\mathcal{P}$ so that the following properties hold:\\
$\mathrm{(a)}$ $f(z,\underline{\omega})=P(z,\underline{\omega})g(z,\underline{\omega})$ for any $(z,\underline{\omega})\in \mathcal{D}(z_{0},r)\times\mathcal{P}$,\\
$\mathrm{(b)}$ $g(z,\underline{\omega})\neq 0 $ for any $(z,\underline{\omega})\in \mathcal{D}(z_{0},r)\times\mathcal{P}$,\\
$\mathrm{(c)}$ for any $\underline{\omega}\in\mathcal{P}$, $P(\cdot,\underline{\omega})$ has no zeros in $\mathbb{C}\backslash\mathcal{D}(z_{0},r)$.
\end{lemma}

\subsection{Implicit function theorem}
\begin{lemma}\cite[Lemma 4.2]{GSV19-Inventiones}\label{implicit}
Let $f(z,\omega)$ be an analytic function defined on the polydisk
$$\mathcal{P}=\{(z,\omega)\in \mathbb{C}\times\mathbb{C}^{n}: |z|, |\omega|<\rho_{0}\}.$$
Let $M_{1}=\sup|\partial_{z}f|$, $M(2)=\mathop{\max}\limits_{|\alpha|=2}\sup|\partial^{\alpha}f|$. Assume that $f(0,0)=0$, $\mu_{0}:=|\partial_{z}f(0,0)|>0$. Let
$$\rho_{1}\leq\min(\rho_{0}/2,c(n)\mu_{0}M(2)^{-1}),\quad r_{i}=c(n)\rho_{1}\min(1,\mu_{0}/|\partial_{\omega_{i}}f(0,0)|),$$
with $c(n)$ a sufficiently small constant. Then for any $\omega$, $|\omega_{i}|<r_{i}$, the equation
$$f(z,\omega)=0$$
has a unique solution $|z(\omega)|<\rho_{1}$ which is an analytic function of $\omega$.
\end{lemma}

\end{appendices}

\vskip1cm

\noindent{$\mathbf{Acknowledgments}$}

This work was supported in part by the  NSFC (No. 11571327, 11971059).

\end{document}